\documentclass[a4paper]{amsart}

\usepackage{pgf,tikz}
\usepackage{amssymb}
\usepackage{amsmath}
\usepackage{amsthm}
\usepackage{graphicx}
\usepackage{fancyhdr}
\usepackage{bbm}
\usepackage{multirow}

\usepackage{caption}
\usepackage{subcaption}
\usepackage{bbm, dsfont}
\usepackage{stmaryrd}
\usepackage{algorithm}
\usepackage{algorithmic} 
\usepackage{mathrsfs}
\def\P{{\mathbb P}}

\def\R{{\mathbb R}}

\def\N{{\mathbb N}}

\def\T{{\mathcal{T}}}

\def\barmun{\bar {\mu}^n}
\def\barmunl{\bar {\mu}^{n,+}}
\def\barmunr{\bar {\mu}^{n,-}}

\def\barmu{\bar\mu}
\def\barmul{\bar\mu^{+}}
\def\barmur{\bar\mu^{-}}

\def\T2a{{\tau_{2\alpha^{+}^{+}}}}
\def\t2a{{t_{2\alpha^{+}^{+}}}}
\def\d {\text{ d}}
\newcommand\cro[1]{\left\langle #1 \right\rangle}

\newcommand\procn[1]{\left(#1,\,t \in \llbracket 0,n \rrbracket\right)}
\newcommand\procu[1]{\left(#1,\,t \in [ 0,1 ]\right)}

\def\\xi{{\cal{\xi}}}

\def\M{{\mathcal M}}

\def\({{\Bigl(}}
\def\){{\Bigr)}}
\newcommand\pr[1]{{\mathbb P}\left[#1\right]}
\newcommand\esp[1]{{\mathbb E}\left[#1\right]}

{\bf}{\it}
{\bf}{\it}
\newtheorem{theorem}{Theorem}
\newtheorem{definition}{Definition}

\newtheorem{remark}{Remark}
\newtheorem{proposition}{Proposition}

\begin{document}




\title[Markovian online bipartite matching]{Markovian online matching algorithms on large bipartite random graphs}
	\author[Diallo Aoudi, Moyal and Robin]{Mohamed Habib Aliou Diallo Aoudi, Pascal Moyal and Vincent Robin}


\begin{abstract}
In this paper, we present an approximation of the matching coverage on large bipartite graphs, for {\em local} online matching algorithms based on the sole knowledge of the remaining degree of the nodes of the graph at hand. This approximation is obtained by applying the Differential Equation Method to a measure-valued process representing an alternative construction, in which the matching and the graph are constructed simultaneously, by a uniform pairing leading to a realization of the bipartite Configuration Model. The latter auxiliary construction is shown to be equivalent in distribution to the original one. It allows to drastically reduce the complexity of the problem, in that the resulting matching coverage can be written as a simple function of the final value of the process, and in turn, approximated 
      by a simple function of the solution of a system of ODE's. By way of simulations, we illustrate the accuracy of our estimate, and then compare the performance of an algorithm based on the minimal residual degree of the nodes, to the classical greedy matching. 
      \end{abstract}



\maketitle

%


\section{Introduction}
\label{sec:intro} 

Maximal matching on graphs is a classical problem in graph theory. Its applications range from simple scheduling to online advertisement, from jobs or housing allocations to the sequencing of protein chains, and so on. 
Long after the first polynomial optimal solution, proposed by Edmonds \cite{Blossom}, this problem has still received a significant attention, especially over the last three decades, due to the rise of large networks, leading to a flourish of 
new applications (in e-commerce, for example), forcing a reexamination of the tools previously used. 
As a consequence, two main new approaches emerged.

In 1990, Karp et al. \cite{Karp} introduced the concept of \emph{Online bipartite matching}. In this paradigm, the graph is split into \emph{customers} and \textit{items}. Customers arrive one by one, and it is the role of the matching algorithm to pick one of the items for a match. These assumptions lead to a natural tradeoff between optimality in the matching size and cheaper costs. In \cite{Karp}, Karp et al. used the \emph{adversarial} order of arrivals as a worst case scenario  for the matching completion, leading to the now classical $1-1/e$ bound for the {\em matching coverage}, namely, the  ratio of matched items out of the total set. This ratio is achieved whenever the nodes are drawn repeatedly from a known distribution, and the matching algorithm is greedy - see below. The construction of algorithms that are able to beat this benchmark under various conditions, 
      has then been the subject of an important line of research, see e.g.  \cite{goel2008online,feldman2009online} and references therein. More recently, various extensions of the online bipartite matching problem have been proposed among which, stochastic matching \cite{StoMatch} (meaning that each edge emanating from the online nodes exists with a given probability), random customer arrivals \cite{RandArr}, or models with patience times \cite{Pati}. Bounds for matching algorithms on (general) random regular graphs of fixed girth are given in 
      \cite{gamarnik2010randomized}, using a randomized version of the greedy matching algorithm.

The second approach is more typically computer-scientific. \emph{Streaming algorithms} are algorithms that only process a {stream}, on a limited part of a given graph (for memory constraints for example). They arise as a natural solution while dealing with large data sets and several authors have worked to apply them to the matching problem (see e.g. \cite{chenstre}, \cite{semistream} and \cite{deterstream}, and references therein). This line of research mainly addressed {the type of input} for the stream (either dynamic or enter only), its {randomness} or its {relative size} to the whole graph.

Most references above provide bounds for the performance of given algorithms, that are supported by combinatorial arguments.
However, precisely assessing the performance of a given matching algorithm as the size of the graphs gets large, is 
a challenging problem. Indeed, the analysis of the asymptotic of a sequence of graphs, and matchings on those graphs becomes less and less tractable as their size grows large. 
It can then be preferable to {\em approximate}  the matching size as a function of a scalable statistic of the graph.  
Following this line of thought, in this paper we propose an alternative approach to the bipartite matching problem: 
assessing the performance of given matching algorithms through a deterministic {approximation} of a stochastic process representing the matching procedure, in the large graph asymptotic. The general idea is as follows: 
rather than precisely defining and keeping track of the whole geometry of the graph 
(an information that is in general unavailable for large graphs), we generate a graph from its degree distribution. To this end, we use a classical uniform pairing procedure of the {\em half-edges} of the nodes, see a precise definition below. This construction leads to the so-called Configuration model, namely, a uniformly drawn realization of a graph having the prescribed degree distribution. 
See \cite{bollobas2001random} and \cite{RDH} for the main properties of configuration models, and \cite{OlvChen} 
for the extension to the case of bipartite and oriented graphs. 
In parallel, we construct simultaneously an online matching 
(meaning in this context, that each edge that is added to the matching cannot be erased afterwards) on the resulting graph. 
As a main interest, this simultaneous construction leads to a simple Markov representation, keeping track of the {\em remaining degrees} of the nodes that are not yet fully attached to the graph (a definition that will be made precise hereafter), provided that the matching algorithm depends only on these remaining degrees. We say in that case that the matching algorithm is {\em local}. 
The underlying Markov process is then measure-valued, where each measure is a sum of Dirac masses 
marking the remaining degrees of the nodes. By doing so, we do not need to keep track of the precise form of the constructed graph, and in fact, we do not need to have access to it. Then, using usual approximation tools for Markov processes, one can identify the approximation of the considered process as the solution of an ordinary differential equation. This results in a generalization, for measure-valued processes, of the celebrated {\em Differential Equation Method}, see Wormald \cite{WCours}. By doing so, we retrieve an estimate of the resulting matching coverage as a simple function of the latter (deterministic) solution, without knowledge of the precise geometry of the graph at hand. 
Remarkably, we show hereafter that the resulting matching coverage has the same distribution as the one obtained when applying the corresponding online, and local, matching algorithm on a previously constructed bipartite graph $\tilde G$, conditional on the fact that the resulting graph constructed by the CM is precisely $\tilde G$. 
As a consequence, our estimate of the matching coverage by the differential equation method, 
provides a remarkably accurate estimation of the matching coverage of the considered local algorithm on a given bipartite graph, a result that we support with extensive simulations. 

The extension of the differential equation method to measure-valued processes, resulting from a simultaneous construction of the CM and an exploration algorithm on the latter, first appeared in \cite{decreusefond2012}, to describe the propagation of an SIR epidemics on an heterogenous graph. Then a closely related idea was applied in \cite{Jam}, to approximate the size of maximal independent sets on graphs with given degree distributions (for a more direct use of the differential equation method on the same topic, see also \cite{Jan}). This led to various extensions, to address various coverage problems in CSMA-type algorithms for radio-mobile and ad-hoc communication networks, see 
\cite{bermolen2016estimating}. 

Before that, measure-valued processes Markov processes were first introduced in the queueing literature. Space of measures are amenable to showing weak convergence of sequence of processes under an appropriate scaling, and the exhaustive representation of queueing systems by point measures, in which each Dirac mass typically represents the characteristic of a customer in line, led to fruitful developments 
      in fluid and diffusion approximations of the systems at hand, see e.g. \cite{gromoll2002fluid} for processor sharing queues, \cite{doytchinov2001real,decreusefond2008fluid} for queueing systems with impatient customers, 
      \cite{decreusefond2008functional} for infinite-server queues, or \cite{kaspi2011law} for many-server queues. 
      
      This paper is organized as follows. After a preliminary section, in Section \ref{sec:explo} we define the local matching algorithms we will address hereafter, among which, the {\em minimal residual} algorithm, which can be seen as an analog of the Degree-greedy algorithm defined in \cite{bermolen2019degree} for the construction of maximal independent sets on random graphs.  
       In Section \ref{sec:CM}, we detail our joint construction of a random (multi-)graph by the CM, and of an online matching on the latter. As will clearly appear below, the local matching algorithms then work as some kind of tagging processes,  alongside the uniform-pairing construction process of the configuration model. In Section \ref{sec:measures}, we introduce the measure-valued Markov processes representing the two previous constructions. In particular, Theorem \ref{thm:coupling} presents a simple connexion between the two processes. In Section \ref{sec:fluid}, we introduce our {\em hydrodynamic approximation}, namely, the limiting system of Ordinary Differential Equations approximating the dynamics of the measure-valued process representing the construction of Section \ref{sec:CM}. 
       Finally, in Section \ref{sec:simu} we present our simulation results. We first illustrate the accuracy of our approximation for various degree distributions, and then compare the matching coverage obtained by the minimal residual algorithm to that resulting from the greedy algorithm. We also briefly compare our results with the results of Karp et al. in \cite{Karp}.  All our code and data are available upon request.

\section{Preliminary}
We denote respectively by $\R$, $\R_+$, $\N$ and $\N^*$, the sets of real numbers, non-negative real numbers, of non-negative and positive integers. For any finite set $A$, we write $|A|$ for the cardinality of $A$. 
For any $a\le b\in \N$, we let $\llbracket a,b \rrbracket$ denote the integer interval $\{a,a+1,...,b-1,b\}$. 
We let $\M(\N)$ be the set of integer finite measures on $\N$. Denote for any $\mu$ in $\M(\N)$ and any 
$\varphi:\R\to\R$ such that the following is well defined, 
\[\cro{\mu,\varphi} =\int_{\N}\varphi(x) \d \mu(x)=\sum_{i\in\N}\varphi(i)\mu(i).\] 
Let $\mathcal B_b(\N)$ be the space of bounded functions: $\N\to \R$. We denote by $\mathbf 1$,  the function $\R\to \R$ that is identically equal to 1, 
$X$ the identity function on $\R$ and $X^2(x)=x^2$ for all 
$x\in\R$. As a consequence, for any $\mu\in\M(\N)$, $\cro{\mu,\mathbf 1}$, $\cro{\mu,X}$ and $\cro{\mu,X^2}$ denote respectively the total mass, the first and the second moment of $\mu$, if any. Denote also for all $\varphi:\N\to\R$, 
      by $\Delta\varphi$ the discrete gradient of $\varphi$, defined by 
      \[\Delta\varphi(x)=\varphi(x)-\varphi(x-1),\, x \in \N.\] 

All random variable are defined on a common probability space $(\Omega,\mathcal F,\P)$. 
The space $\M(\N)$ is endowed with its weak topology, in other words $\mu^n \longrightarrow \mu$ in $\M(\N)$ whenever 
$\cro{\mu^n,\varphi} \longrightarrow \cro{\mu,\varphi}$ for all bounded $\varphi$. Then it is well known that $\M(\N)$ is a Polish space, see e.g. \cite{Bill}. 

For any Polish space $E$, we denote by $\mathrm D([0,1],E)$ the space of \textsl{cadlag} 
(i.e. right-continuous, with left-hand limits at all points) $E$-valued processes on $[0,1]$, equipped 
with the Skorokhod topology. 

\section{Local matching algorithms on a graph}
\label{sec:explo} 
{We start by introducing the general class of matching algorithms we will consider hereafter.}


\subsection{Matching through exploration on a graph}
{We now define the class of {\em local} matching algorithms we consider in this work, in a sense to be specified below. 
We fix a (non-oriented) bipartite graph $G=(V,E)$, where $V=V^+\cup V^-$ is the set of nodes and $E\subset V^+\times V^-$ is the set of edges. 
We assume that the bipartition of the graph is balanced, i.e., 
$|V^+|=|V^-|=n$ for some positive integer $n$. 
At any time $t$, we are given two disjoint subgraphs of $G$:
\begin{itemize}
\item $\tilde G_t=(\tilde U_t, \tilde E_t)$ 
represents the {\em remaining graph} of our construction at $t$. The nodes of $\tilde U_t$ are said {\em undetermined} at $t$. 
As will appear clearly in the construction below, the graph $\tilde G_t$ is itself bipartite (as a subgraph of a bipartite graph), and we denote by 
$\tilde U_t=\tilde U_t^+\cup \tilde U^-_t$, the bipartition of its nodes.  For any $\tilde v \in \tilde U_t$ we denote by $N_t(\tilde v)$ the set of neighbors of $\tilde v$ at $t$, 
and by $d_t(\tilde v)=|N_t(\tilde v)|$ the {degree} of $\tilde v$ at $t$, namely, the number of its neighbors in $\tilde G_t$. 
\item $\tilde G'_t=(\tilde M_t,\tilde E'_t)$ is the {\em matching} at time $t$. This is also a (bipartite) subgraph of $G$ in which all nodes are of degree one. 
The set $\tilde M_t:=\tilde M^+_t \cup \tilde M^-_t$ gathers the {\em matched} nodes at $t$.
\end{itemize}}
We also let $\tilde I_t$ be the set of {\em isolated} nodes at $t$. These are nodes that won't appear in the final matching, 
and we have the disjoint union $V=\tilde U_t \cup \tilde M_t\cup \tilde I_t$. 

At first, all nodes are undetermined, i.e. we set $\tilde U_0=V$, $\tilde M_0=\emptyset$ and $\tilde I_0=\emptyset$. 
We also set $\tilde E_0=E$ and $\tilde E'_0=\emptyset$, in a way that 
$\tilde G_0=G$ and $\tilde G'_0=(\emptyset, \emptyset)$. 
The matching algorithm then proceeds by induction, as follows: At any time $t\in \llbracket 0,n-1\rrbracket$, 
\begin{enumerate}
\item[{\bf Step $\tilde 1$.}] A vertex $\tilde J^+$ is chosen {uniformly at random in $\tilde U^+_t$.} 

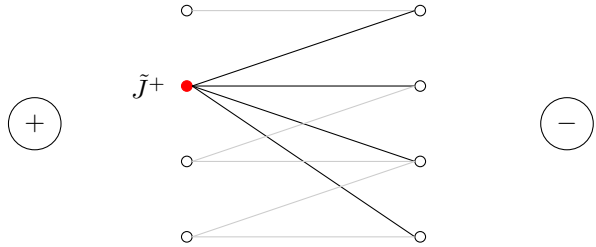
\begin{figure}[h!]			
			\begin{center}
\begin{tikzpicture}
\coordinate (p1) at (0,5);
\coordinate (p2) at (0,4);
\coordinate (p3) at (0,3);
\coordinate (p4) at (0,2);
\coordinate (m1) at (3,5);
\coordinate (m2) at (3,4);
\coordinate (m3) at (3,3);
\coordinate (m4) at (3,2);
\draw (p1) circle (0.07);
\draw[color=red,fill] (p2) circle (0.07);
\draw (p2)+(-0.5,0) node {$\tilde{J}^{+}$};
\draw (p3) circle (0.07);
\draw (p4) circle (0.07);
\draw (m1)+(0.07,0) circle (0.07);
\draw (m2)+(0.07,0) circle (0.07);
\draw (m3)+(0.07,0) circle (0.07);
\draw (m4)+(0.07,0) circle (0.07);
\draw[color=gray!40] (p1)+(0.07,0) to (m1);
\draw[color=black] (p2)+(0.07,0) to (m1)+(-0.07,0);
\draw[color=black] (p2)+(0.07,0) to (m2)+(-0.07,0);
\draw[color=black] (p2)+(0.07,0) to (m3)+(-0.07,0);
\draw[color=black] (p2)+(0.07,0) to (m4)+(-0.07,0);
\draw[color=gray!40] (p3)+(0.07,0) to (m2)+(-0.07,0);
\draw[color=gray!40] (p3)+(0.07,0) to (m3)+(-0.07,0);
\draw[color=gray!40] (p4)+(0.07,0) to (m3)+(-0.07,0);
\draw[color=gray!40] (p4)+(0.07,0) to (m4)+(-0.07,0);
\node[circle,draw] at (-2,3.5) {$+$};
\node[circle,draw] at (5,3.5) {$-$};
\end{tikzpicture}
			\label{fig:S1}
			\caption{Exploring a vertex of a bipartite graph}
			\end{center}		
\end{figure}

\item[{\bf Step $\tilde 2$.}] We select a match for $\tilde J^+$ if this is feasible at all: 
\begin{enumerate}
\item[{\bf Step $\tilde 2$a)}] If $\tilde J^+$ has at least one neighbor in $\tilde U_t^-$, i.e., if the set $N_t(\tilde J^+)$ is non-empty, 
then it is the role of the matching criterion to select one of these, say $\tilde J^-$, to be the {\em match} of $\tilde J^+$. 
We assume that the criterion is {\em local} in the sense that it only depends on the degrees of the neighbors of $\tilde J^+$, and possibly on a 
random choice that is independent of every other random variables so far. To formalize this, fix a deterministic mapping 
\begin{equation}\label{eq:defphi}
\Phi:\begin{cases}\displaystyle\bigcup_{k=1}^n \,(\N^*)^k &\longrightarrow \llbracket 1,n \rrbracket \\
				\displaystyle(x_1,...,x_k) &\longmapsto j \in \llbracket 1,k \rrbracket
				\end{cases}.\end{equation}
Then, let $k:=|N_t(\tilde J^+)|\le |\tilde U^-_t|$,
denote by $\left\{\tilde J^-_{1},...,\tilde J^-_{k}\right\}$, the elements of $N_t(\tilde J^+)$, and draw a permutation 
$\sigma\in\mathfrak S_{k}$ uniformly at random. The permutation $\sigma$ is possibly used to draw a list of {priorities} between elements of $N_t(\tilde J^+)$. 
Then, we set $\tilde J^-:=\Phi(\tilde J^+):=\tilde J^-_{\sigma(j)}$, where 
$j=\Phi\left(d_t(\tilde J^-_{\sigma(1)}),...,d_t(\tilde J^-_{\sigma(k)})\right)$. 
{In Figure \ref{fig:choix}, the \textit{match} is represented in blue for two examples of local matching criteria: {\sc greedy} and {\sc minres}, properly defined in Definition \ref{def:algo}.}
\begin{figure}[!h]
			\centering
			\includegraphics[scale=0.24]{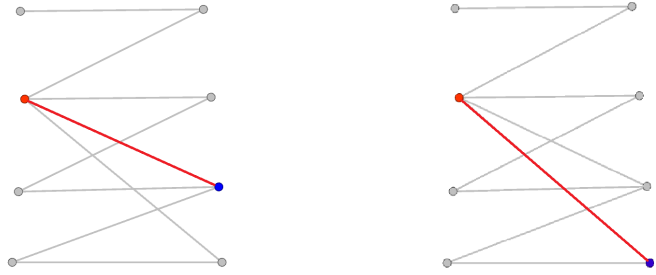}
			\caption{Selecting the match for the two algorithms {\sc greedy} (left) and {\sc minres} (right), defined in Definition \ref{def:algo}.}
			\label{fig:choix}		
\end{figure}
Then the matched nodes $\tilde J^+$ and $\tilde J^-$ are deleted from the graph $\tilde G_{t}$, and added to the matching $\tilde G'_t$. 
Specifically, we set 
\[\left\{\begin{array}{lll}
\tilde U_{t+1}&=\tilde U_t\setminus\{\tilde J^+,\tilde J^-\},\quad &\tilde G_{t+1}=\mbox{Induced subgraph of }\tilde U_{t+1}\mbox{ in }\tilde G_t,\\
\tilde M_{t+1}&=\tilde M_t\cup\{\tilde J^+,\tilde J^-\},\,\quad &\tilde E'_{t+1}=\tilde E_t\cup\{\{\tilde J^+,\tilde J^-\}\},\\
\tilde I_{t+1}&=\tilde I_t.&
\end{array}\right.\]
Figure \ref{fig:finaux} shows the resulting graph $\tilde G_{t+1}$. 
\begin{figure}[!h]
			\begin{center}
			\includegraphics[scale=0.3]{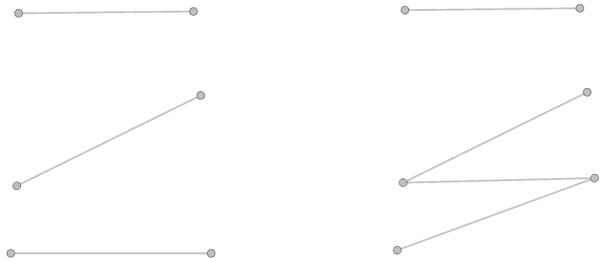}
			\caption{Resulting graph $\tilde G_{t+1}$ after the deletion of the matched vertices, for {\sc greedy} (left) and {\sc minres} (right)}. 
			\label{fig:finaux}
			\end{center}
\end{figure}
\item[{\bf Step $\tilde 2$b)}] If $\tilde J^+$ has no neighbor in $\tilde U_t^-$, i.e., $N_t(\tilde J^+)=\emptyset$, then 
$\tilde J^+$ is just added to the set of isolated nodes, i.e. we set 
\begin{equation*}
\tilde U_{t+1}=\tilde U_t,\quad \tilde E_{t+1}=\tilde E_t,\quad
\tilde M_{t+1}=\tilde M_t,\quad \tilde E'_{t+1}=\tilde E_t,\quad \tilde I^+_{t+1}=\tilde I^+_t\cup \{\tilde I^+\}.
\end{equation*}
\end{enumerate}
\item[{\bf Step $\tilde 3$.}] We set $t:=t+1$ and go back to step 1. 
\end{enumerate}
At the terminating point $t=n$ of the procedure, we necessarily get that 
$|\tilde M^+_n|=|\tilde M^-_n|$ and $\tilde U^+_n=\emptyset$. 
Notice that, if $|\tilde I^+_n|\ne 0$, then there are at time $n$, as many undetermined nodes on the right-hand side as of isolated nodes on the left-hand side, namely there are $|\tilde I^+_n|$ of them. We can then set all undetermined nodes on the right-hand side as isolated, since the latter nodes can no longer be attached to the matching, i.e. we set $\tilde I^-_n:=\tilde U^-_n$, and then $\tilde U^-_n := \emptyset$.  
We finally obtain that $$|\tilde M_n|+|\tilde I_n|=2n,$$ in other words all nodes are either matched, or isolated. 
The {\em matching coverage} 
$\tilde{\mathbf M}^n_\Phi(G)$ is then the proportion of initial nodes that ended up in the matching at the termination time $n$, that is, 
\begin{equation}
\label{eq:ratio1}
\tilde{\mathbf M}^n_\Phi(G)={|\tilde M_n|\over 2n}=1-{|\tilde I_n|\over 2n}\, \,\in [0,1].
\end{equation}


\noindent We now define two particular local matching algorithms,
\begin{definition}
\label{def:algo}\rm
\begin{itemize}
\item We say that $\Phi$ is {\em greedy}, and we denote $\Phi=\textsc{greedy}$ if, at step 3a) above $\tilde J^-$ is chosen uniformly at random, independently of all r.v.'s involved up to $t$, within the set $N_t(\tilde J^+)$. Specifically, in (\ref{eq:defphi}) we set $\Phi(x_1,...,x_k)=1$ for any $k\in \llbracket 1,n \rrbracket$ and all $(x_1,...,x_k)\in(\N^*)^k$, 
in a way that $\tilde J^-=\tilde v_{\sigma(1)}$. In other words, the permutation $\sigma$ 
is used to draw, uniformly at random, a list of priorities between the elements of $N_t(\tilde J^+)$, and the priority node $\tilde v_{\sigma(1)}$ is chosen 
as the match of $\tilde J^+$. 
\item We say that $\Phi$ is {\em minimal residual}, and we denote $\Phi=\textsc{minres}$ if, at step 3a) above, we chose 
$\tilde J^-$ uniformly at random, independently of all r.v.'s up to $t$, within the set of nodes of minimal residual degree in $\tilde G_t$ amongst the neighbors 
of $\tilde J^+$. Then, we use again $\sigma$ to construct a list of priorities between the nodes of minimal degree in $\tilde U^-_t$, and in 
(\ref{eq:defphi}) we set 
\[\Phi(x_1,...,x_k)=\min\Bigl\{\ell \in \llbracket 1,k \rrbracket\,:\,\ell \in \mbox{argmin}\left\{x_i\,:\,i\in\llbracket 1,k \rrbracket\right\}\Bigl\}\]
for all $k$ and all $(x_1,...,x_k)\in(\N^*)^k$. 
\end{itemize}
\end{definition}

\begin{remark}
A similar algorithm, called {\em Degree-greedy}, was proposed in \cite{bermolen2019degree} for constructing maximal independent sets on general graphs. 
\end{remark}

%

\begin{algorithm}
\caption{General local matching algorithm for a choice criterion $\Phi$}
\begin{algorithmic}
\REQUIRE Non-empty bipartite graph $G(V^+\cup V^-,E)$
\STATE $\mbox{Matching} \leftarrow \emptyset$;
\FOR{$t < n$}
	\STATE Pick a (uniformly) random vertice $\tilde J^+$ in $\tilde U^+$;
	\STATE Discover its neighborhood $N(\tilde J^+)$ in $\tilde U^-$;
	\STATE $\tilde U^+ \leftarrow \tilde U^+ \setminus \left\{\tilde J^+\right\}$;
	\IF{$N(\tilde J^+)$ is empty}
		\STATE  Do not do anything more
	\ELSE
		\STATE $\tilde J^- \leftarrow \Phi(\tilde J^+)$;
		\STATE $\tilde U^- \leftarrow \tilde U^- \setminus \left\{\tilde J^-\right\}$;
		\STATE $\mbox{Matching} \leftarrow \mbox{Matching} \cup (\tilde J^+,\tilde J^-)$;
	\ENDIF
	
	$t \leftarrow t+1$;
\ENDFOR
\end{algorithmic}
\end{algorithm}

\section{Coupled construction with the Configuration Model}
\label{sec:CM}

The previous matching algorithms are based on an exploration of the bipartite graph under consideration. By their inductive construction, as defined above, they naturally induce a Markovian dynamics. However, throughout the construction, one needs to keep track of the whole topology of the graph. As the size of the graph grows large, this approach has two main limitations: First, it leads to a dramatic increase of the size of the state space of the underlying Markov process - see below. Second, in practical cases, one rarely has access to the whole topology of a large network. Instead, a more restricted statistical information can be made available, such as the distribution of the degrees (i.e. the size of the neighborhoods) of the nodes. 

To address both difficulties at once, in this section we show how to assess the performance of the algorithms under consideration, by introducing a Markovian construction based 
solely on the degree distribution of the nodes, without any knowledge of the precise topology of the graph. 
For doing so, one way is to generate uniformly at random, a graph having a prescribed degree distribution, and to perform simultaneously the matching algorithm under consideration, as described in Section \ref{sec:explo}.

As is well known, a uniform random realization of a graph having a given degree distribution can be obtained via the uniform pairing procedure, inherent to the so-called Configuration Model (CM, for short), see \cite{bollobas2001random,molloy2011critical}. This procedure consists of initially considering that all nodes are detached from the graph, and have as many half-edges as their degree, and then to selecting sequentially, uniformly at random, nodes, and attaching to the graph by pairing their half-edges to other ones, chosen again uniformly at random. It is easily seen that such a procedure in general leads to a multi-graph, i.e. a graph containing self-loops and/or multiple edges; however the number of such ``bad'' edges can be shown to be a $o(1/n)$, or even less \cite{newman2018networks}, see e.g. \cite{RDH} for the CM on general graph, and \cite{OlvChen} for CM on bipartite and oriented graphs.

In this section, the dynamics of the analog of the local algorithms defined in Section \ref{sec:explo} are constructed {\em simulatenously} with the configuration model. 
This joint construction will induce a more simple Markovian dynamics, leading in turn to a simple approximation of the matching coverage of the model, see Section \ref{sec:fluid}. 

\subsection{Local matching algorithms on the CM}

Let $\xi^+$ and $\xi^-$ be two probability measures on $\N$. 
Let $n$ be a positive integer and $\mathbf d^+:=(d^+(1),...,d^+(n))$ and $\mathbf d^-:=(d^-(1),...,d^-(n))$ be two $n$-samples, respectively of the probability distributions 
$\xi^+$ and $\xi^-$. These two samples are conditioned to have the same total mass, i.e., to be such that 
$\sum_{i=1}^nd^+(i)=\sum_{i=1}^n d^-(i).$ Their elements will hereafter represent the degrees of the nodes on both sides of the bipartition, 
see e.g. Figure \ref{fig:ex0} for $n = 4$. 
\begin{figure}[!h]
			\begin{center}
			\includegraphics[scale=0.3]{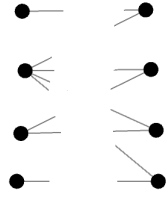}
			\caption{A running example : $\hat\mu_0^+ = 2\delta_1 + \delta_2 + \delta_4$ and $\hat\mu_0^- = 4\delta_2$}
			\label{fig:ex0}
			\end{center}
\end{figure}

We now introduce our joint construction of a multi-graph having degree 
distributions $\mathbf d^+$ and $\mathbf d^-$, and of a matching on the latter. 
We let $V^+=\{v_1^+,...,v_n^+\}$ (resp., $V^-=\{v_1^-,...,v_n^-\}$) be the set of nodes of the graph on the left-hand  (resp., right-hand) side of the bipartition, and view for all $i\in\llbracket 1,n \rrbracket$, $d^+(i)$ (resp., $d^-(i)$) as the {\em degree} of node $v_i^+$ 
(resp. $v_i^-$). We then set for all $i$, 
$a_0(v^+_i)=d^+(i)$ and $a_0(v^-_i)=d^-(i)$, and define the two following sets, 
\[A^+_0=\Bigl\{\bigl(v^+_{1},a_0(v^+_{1})\bigl),\cdots,\bigl(v^+_{n},a_0(n)\bigl)\Bigl\},\quad 
A^-_0=\Bigl\{\bigl(v^-_{1},a_0(v^-_{1})\bigl),\cdots,\bigl(v^-_{n},a_0(v^-_{n})\bigl)\Bigl\}.\]
Observe that we have $\sum_{l=1}^n a_0(v^+_{i_l})=\sum_{l=1}^p a_0(v^-_{j_l})$ by assumption.             
    At first, the graph $G_0$ consists of the set of disconnected nodes $V:=V^+\cup V^-$, i.e., $G_0=(V,\emptyset)$. We also set 
    $M_0^+= M_0^-=\emptyset$, $U_0^+=V^+$ and $U^-_0=V^-$, and let $G'_0$ be the empty graph, i.e. $G'_0=(\emptyset,\emptyset)$. 
    We then proceed by induction, as follows. At any time $t$, we are given: 
\begin{itemize}
\item A bipartite multi-graph $G_t=(V,E_t)=(M_t\cup U_t\cup I_t,E_t)$ representing the partially constructed connexions between elements of $V$, 
where we denote by :
\begin{itemize}
\item $M_t^+$ (resp., $M_t^-$) the set of {\em matched} nodes at $t$ on the left-hand (resp., right-hand) side at $t$, which are nodes that are fully attached to the graph at $t$, and belong to the matching at $t$; 
\item $U_t^+$ (resp., $U_t^-$), the set of {\em undetermined} nodes at $t$ on the left-hand (resp., right-hand) side, that is, nodes that do not belong to the matching at $t$, but can still be attached to it; 
\item $I_t^+$, the set of {\em isolated} nodes at $t$ on the left-hand side, that is, nodes that are already fully attached to the graph at $t$, but do not belong to the matching at $t$. 
\end{itemize}
We also set $M_t=M^+_t\cup M^-_t$ and $U_t=U^+_t\cup U^-_t$. 
By our very construction, all nodes of $M_t$ (if any) will have degree at least one in 
$G_t$, and we let $U^+_t=\{v^+_{i_1},...,v^+_{i_p}\}\subset V^+$ and $U^-_t=\{v^-_{j_1},...,v^-_{j_r}\}\subset V^-$. (We skip the dependance in $t$ in the parameters $v^{+/-}_{j_l}$, for short.) 
\item A perfect {\em matching} $G'_t=(M_t,E'_t)$ 
       on the induced subgraph of $M_t$ in $G_t$. In particular, 
      $E'_t$ is a set of subsets of pairs of $M_t$ of the form $\{v^+_i,v^-_j\}$ for $v_i^+\in M^+_t$ and $v_j^-\in M_t^-$, such that any element of $M_t$ appears in exactly one pair of $E'_t$. 
\item A couple of sets of pairs 
\[A^+_t=\Bigl\{\bigl(v^+_{i_1},a_t(v^+_{i_1})\bigl),\cdots,\bigl(v^+_{i_p},a_t(v^+_{i_p})\bigl)\Bigl\},\quad 
A^-_t=\Bigl\{\bigl(v^-_{j_1},a_t(v^-_{j_1})\bigl),\cdots,\bigl(v^-_{j_n},a_t(v^-_{j_n})\bigl)\Bigl\},\]
such that $\sum_{l=1}^n a_t(v^+_{i_l})=\sum_{l=1}^p a_t(v^-_{j_l})$ (possibly understanding sums over empty sets as 0). 
For any $l \in \llbracket 1,p \rrbracket$ we interpret $a_t(v^+_{i_l})$ as the {\em availability} 
of node $v^+_{i_l}$ at $t$, that is, the number of {available half-edges} (towards the right-hand side) of $v^+_{i_l}$, and likewise for the right-hand side. See an example on Figure \ref{Fig:ex1}. 
\end{itemize}
\begin{figure}[!h]
			\begin{center}
			\includegraphics[scale=0.3]{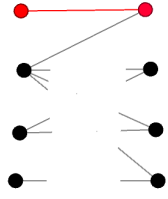}
			\caption{State at iteration $t$ : the matched nodes (elements of $G'_t$) are  red. The undetermined nodes (labelled from top to bottom) are black, and we get 
			$A_t^+=\left\{(v^+_{i_1},3),(v^+_{i_2},2),(v^+_{i_3},1)\right\}$; 
			$A_t^-=\left\{(v^-_{j_1},2),(v^-_{j_2},2),(v^-_{j_3},2)\right\}$.}
			\label{Fig:ex1}
			\end{center}
\end{figure}
\noindent Then we proceed as follows: at any time $t\in \llbracket 0,n-1 \rrbracket$, 
  \begin{enumerate}
  \item[{\bf Step 1.}] We select an uncompleted node on the left-hand uniformly at random, and then 
complete its neighborhood in the graph. 
Specifically, we first draw the realization of a r.v. of uniform distribution 
$\mathscr U\left(\llbracket 1,p \rrbracket\right)$. Say the outcome is $q$. Then, we set $J^+=v^+_{i_q}$, see Figure \ref{Fig:ex2}.
  \begin{figure}[!h]
			\begin{center}
			\includegraphics[scale=0.3]{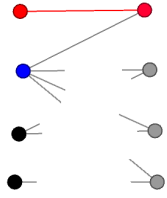}
			\caption{$J^+$ chosen at iteration t}
			\label{Fig:ex2}
			\end{center}
  \end{figure}
We then complete the $a_t\left(J^+\right)$ half-edges of $J^+$, if any, into edges by joining them to half edges picked uniformly at random among the half-edges on the right-hand side. Specifically, 
	\begin{enumerate}
	  \item[{\bf Step 1a)}] 
	If $a_t(J^+)>0$, we draw $a_t\left(J^+\right)$ elements uniformly at random among $n$ ``bunches'' of elements of respective sizes $a_t\left(v^-_{j_l}\right)$, $l \in \llbracket 1,n \rrbracket$, 
to construct the emanating edges of node $J^+$ toward the `-' side, and thereby, its neighbors on the opposite side. We let $\{J^{-}_1,...,J^{-}_k\} \subset \left\{v^-_{j_1},...,v^-_{j_n}\right\}$, where $k \le  a_t\left(J^+\right)$, be the set of the $k$ neighbors of $J^{+}$, i.e., the indexes of the bunches containing the chosen half-edges. Note that this operation may lead to multiple edges, whenever several elements of the same bunch of half-edges are chosen, and in this case we have 
$k <  a_t\left(J^+\right)$. 
For all $l \in \llbracket 1,k \rrbracket$, we let $E^{-}_l$ be the number of {edges} shared by $J^{-}_l$ with $J^{+}$, that is, the number of elements in the bunch $J^{-}_l$ chosen in the uniform pairing procedure. Then the neighborhood of $J^+$ is completed, 
      see Figure \ref{Fig:ex3}. 
\begin{figure}[!h]
			\begin{center}
			\includegraphics[scale=0.3]{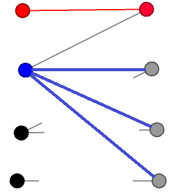}
			\caption{The neighbors of $J^+$ are discovered. 
			}
			\label{Fig:ex3}
			\end{center}
\end{figure} 
   \item[{\bf Step 1b)}] If $a_t(J^+)=0$, then node $J^+$ has no available half-edge at this point. 
Its neighborhood in the graph has thus already been fully determined in the previous steps, 
but this node cannot be added to the matching at $t$. This node then becomes isolated (in the sense that it won't be added to the matching), and we set 
\[I^+_{t+1}=I^+_t \cup \{J^+\},\quad U_{t+1}=U_t\setminus \{J^+\}.\]
  \end{enumerate}
\item[{\bf Step 2.}] 
We determine the match of $J^+$ within the set $\{J^{-}_1,...,J^{-}_k\}$ if the latter is non-empty, and then we complete the neighborhood of the latter node in the graph. Specifically, 
	\begin{enumerate}
	\item[{\bf Step 2a)}] If $a_t(J^+)>0$, then the node $J^+$ will indeed be attached to the matching at $t+1$, and so it is not an isolated node. We then set $I^+_{t+1}=I^+_t.$
Then, we determine the match of $J^+$ on the right-hand side of the graph by a given matching procedure, as defined above. For this, we first draw a permutation 
$\sigma \in \mathfrak S_k$ uniformly at random, and then set $J^-=J^-_{\sigma(m)}$, where $m=\Phi\left(a_t\left(J^-_{\sigma(1)}\right),\,...\,,a_t\left(J^-_{\sigma(k)}\right)\right)$, for a mapping $\Phi$ of the form (\ref{eq:defphi}). See Figure \ref{Fig:ex4}. 
Then we add both nodes $J^+$ and $J^-$ together with 
 the edge $\{J^+,J^-\}$ to the matching $G'_{t+1}$, that is, we set 
 $$M^+_{t+1}=M^+_t\cup \{J^+\},\quad M^-_{t+1}=M^-_t\cup \{J^-\}\quad\mbox{ and }\quad E'_{t+1}=E_t\cup \bigl\{\{J^+,J^-\}\bigl\}.$$
   \begin{figure}[!h]
			\begin{center}
			\includegraphics[scale=0.3]{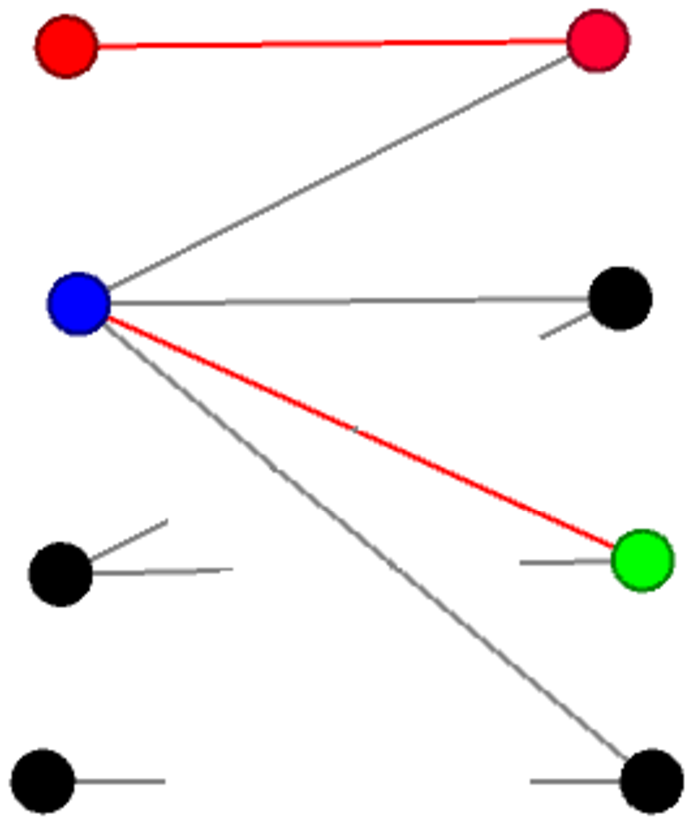}
			\caption{\mbox{choice} of a Match {\color{green}$J^-$} }
			\label{Fig:ex4}
			\end{center}
\end{figure} 

	Finally, we determine the other neighbors of 
	$J^-$ on the left-hand side. Specifically, we are in the following alternative: 
         \begin{itemize} 
         \item If $a_t(J^-)>E^-_{\sigma(m)}$, the node $J^-=J^-_{\sigma(m)}$ still has $a_t(J^-)-E^-_{\sigma(m)}$ uncompleted half-edges. 
      We then draw at random the indexes $J^{+}_1,...,J^{+}_\ell$ of the neighbors of $J^{-}$ on the left-hand side other than $J^{+}$, according to the same uniform pairing procedure as in Step 1a). Namely, we draw $a_t\left(J^-\right) - E^-_{\sigma(m)}$ elements uniformly at random among the $p-1$ ``bunches'' of elements of respective sizes $a_t\left(v^+_{i_l}\right)$, $l \in \llbracket 1,r \rrbracket\setminus \{q\}$ (that is, all unmatched elements on the left-hand side now that $J^+=v^+_{i_q}$ has been matched), to determine the other neighbors of $J^-$ on the left-hand side. 
      We let $\{J^{+}_1,...,J^{+}_\ell\} \subset \left\{v^+_{i_1},...,v^+_{i_p}\right\}\setminus\{i^+_q\}$ (for $\ell \le  a_t\left(J^-\right) - E^-_{\sigma(m)}$) 
      be the set of the $\ell$ neighbors of $J^{+}$, i.e., the indexes of the bunches containing the chosen half-edges. 
      For all $l \in \llbracket 1,\ell \rrbracket$, we let $E^{+}_l$ be the number of {edges} shared by $J^{+}_l$ with $J^{-}$, that is, the number of elements in the bunch $J^{+}_l$ chosen in the uniform pairing procedure. 
      \item If $a_t(J^-)=E^-_{\sigma(m)}$ (which is necessarily the case if $r=1$), then $J^-$ has no more open half-edges to complete.  
      In this case we do not do anything at this stage, and just set $\ell=0$. 
      \end{itemize} 
      In all cases, the neighborhoods of both nodes $J^+$ and $J^-$ are now complete, and we set 
      $$E_{t+1}=E_t\,\cup\, \bigl\{\{J^+,J^-_1\},...,\{J^+,J^-_k\}\bigl\}\,\,\cup\,\,\bigl\{\{J^+_1,J^-\},...,\{J^+_\ell,J^-\}\bigl\},$$ 
      where the second set on the right-hand side above is understood as empty if $\ell=0$. 
      We also update the sets $A^+_t$ and $A^-_t$ by deleting the pairs corresponding to the newly matched nodes $J^+$ and $J^-$, and by 
      updating the remaining number of open half-edges of the unmatched nodes connected to the two newly 
      matched ones, if any. In other words, we set 
       \begin{equation}
    \label{eq:dynA}
    \left\{\begin{array}{lll}
    A_{t+1}^+ &= A_t^+ &\cup \left\{\left(J^+_1,a_t\left(J^+_1\right) - E^+_1\right),\cdots \left(J^+_\ell,a_t\left(J^+_\ell\right) - E^+_\ell\right)\right\}\\
    & &\setminus \left\{(J^+,a_t(J^+)),\left(J^+_1,a_t\left(J^+_1\right) \right),\cdots \left(J^+_\ell,a_t\left(J^+_\ell\right) \right)\right\};\\\\
    A_{t+1}^- &= A_t^- &\cup \left\{\left(J^-_1,a_t\left(J^-_1\right) - E^-_1\right),\cdots \left(J^-_k,a_t\left(J^+_k\right) - E^-_k\right)\right\}\\
    & &\setminus \left\{(J^-,a_t(J^-)),\left(J^-_1,a_t\left(J^-_1\right) \right),\cdots \left(J^-_k,a_t\left(J^+_k\right) \right)\right\}.
    \end{array}\right.
    \end{equation}
    See Figure \ref{Fig:ex5}. 
      \begin{figure}[!h]
			\begin{center}
			\includegraphics[scale=0.3]{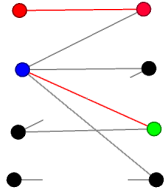}
			\caption{Determining the neighbors of the Match.}
			\label{Fig:ex5}
			\end{center}
\end{figure} 
	\item[{\bf Step 2b)}] If $a_t(J^+)=0$, then no node is added to the matching, and we just set 
	\[G'_{t+1}=G'_t,\quad 
	E_{t+1} =E_t,\quad A_{t+1}^+ =A^+_t\quad \mbox{ and }\quad A_{t+1}^-=A^-_t.\]
	\end{enumerate}
	\item[{\bf Step 3.}] We set $t:=t+1$ and go back to step 1. 
\end{enumerate}

\medskip

\noindent The procedure terminates at time $n$. At that time, we end up with a bipartite multi-graph $G:=G_n=(V,E_n)$, since all half-edges have been completed. 
We then have the following result,

\begin{proposition}
Suppose that the degree distributions $\mathbf d^+:=(d^+(1),...,d^+(n))$ and $\mathbf d^-(d^-(1),...,d^-(n))$ are graphical, namely, they can lead to a bipartite graph $G$ and set 
\begin{align*}
\mathscr G(\mathbf d^+,\mathbf d^-) &=\left\{\mbox{bipartite graphs $\hat G$ on $V$ having degree distributions $\mathbf d^+$ and $\mathbf d^-$}\right\};\\
g(\mathbf d^+,\mathbf d^-) &=\mbox{Card }\mathscr G(\mathbf d^+,\mathbf d^-).
\end{align*}
Let $G$ be the resulting mutligraph of the construction of Section \ref{sec:CM}. 
Then, for any $\hat G\in \mathscr G(\mathbf d^+,\mathbf d^-)$ we get that 
\begin{equation}
\label{eq:uniform}
\pr{G=\hat G \mid \mbox{$G$ is a graph}}={1\over g(\mathbf d^+,\mathbf d^-)},
\end{equation} 
in other words, conditionally on generating a graph, the construction produces graphs uniformly 
on the set of graphs having degree distributions $\mathbf d^+$ and $\mathbf d^-$. 
\end{proposition}

\begin{proof}
At each step, the procedure of uniform choice of the nodes $J^+$ and possibly $J^-$, in the steps 
1a) and 2a) above, is exactly the uniform pairing procedure, as described in \cite{RDH,WCours}. From the so-called {\em independence property} in \cite{WCours}, the resulting multigraph $G$ is then a bipartite realization of the configuration model. 

The uniform property (\ref{eq:uniform}) follows: specifically, Theorem 2.3 in \cite{OlvChen} show how to properly build a {directed} graph from two given degree distributions while ensuring that these distributions are graphical. It is then shown in Proposition 4.1 in \cite{OlvChen}, that when the considereddistributions admit a variance (so that simple graph realizations of the CM do exist), the directed Configuration Model also preserves the uniform property on simple graphs. All that is left is to use the natural bijection between directed and bipartite graphs (as an example, see the bijection using matrices in \cite{DirVsBip}).
\end{proof}

\medskip

At time $n$, we necessarily get that 
$|M^+_n|=|M^-_n|$ and $U^+_n=\emptyset$. Also, as in the construction of Section \ref{sec:explo}, 
if $|I^+_n|\ne 0$, then there are at time $n$, as many undetermined nodes on the right-hand side as of isolated nodes on the left-hand side. We then set all undetermined nodes on the right-hand side as isolated, i.e., we set $I^-_n:=U^-_n$, and then 
$I_n=I^+_n\cup I^-_n$.  
We again obtain that $|M_n|+|I_n|=2n.$ 
Similarly to (\ref{eq:ratio1}), the {matching coverage} 
${\mathbf M}^n_\Phi(\xi^+,\xi^-)$ is then the proportion of initial nodes that ended up in the matching at the termination time $n$, that is,
\begin{equation}
{\mathbf {M}}^n_\Phi(\xi^+,\xi^-)={|M_n|\over 2n}=1-{|I_n|\over 2n}\cdot
\label{eq:ratio2}
\end{equation}

\section{Measure-valued representations} 
\label{sec:measures}

Fix a matching criterion $\Phi$. To compare the two constructions, we define two measure-valued processes. 
%
For all $0\le t \le n$, we define the following point measures: 
\begin{itemize}
\item We let $\tilde\mu^+_t$ (resp., $\tilde\mu^-_t$) 
be the empirical degree distribution of all undetermined nodes at $t$ on the left-hand 
(resp., right-hand) side in the remaining graph $\tilde G_t$ defined in Section \ref{sec:explo}, that is, 
\begin{equation}
\label{eq:defmutilde}
\tilde\mu^+_t=\sum_{\tilde v^+ \in \tilde U_t^+}\delta_{ d_t(\tilde v^+)},\quad \tilde\mu^-_t=\sum_{{\tilde v^- \in \tilde U_t^-}}\delta_{d_t(\tilde v^-)}.
\end{equation}
\item We let $\tilde\mu^+_t$ (resp., $\tilde\mu^-_t$) 
be the empirical distribution representing the availabilities of all undetermined nodes at $t$ on the left-hand (resp., right-hand) side, as defined in Section \ref{sec:CM}: 
\begin{equation}
\label{eq:defmu}
\mu^+_t=\sum_{v^+ \in U_t^+}\delta_{ a_t(v^+)},\quad \mu^-_t=\sum_{{v^- \in U_t^-}}\delta_{a_t(v^-)}.
\end{equation}
\end{itemize}
Now, observe that in the construction of Section \ref{sec:explo}, the isolated nodes at $t=n$ 
on the right-hand side are precisely the undetermined nodes having degree zero at the end of the construction. All the same, in the construction of Section \ref{sec:CM} the isolated nodes at $t=n$ are the undetermined nodes having degree zero at the final step. 
Second, by construction there are in both cases, as many isolated nodes on both sides 
of the bipartition. Therefore, we get that 
\[|\tilde I_n| = 2\tilde \mu^-_n(\{0\}) \quad \mbox{ and }\quad |I_n| = 2\mu^-_n(\{0\}).\]
All in all, it follows from (\ref{eq:ratio1}) and (\ref{eq:ratio2}), that the two matching coverages are respectively given by 
\begin{align}
\label{eq:ratio1bis}
\tilde{\mathbf M}^n_\Phi(G)&=1-{\tilde \mu^-_n(\{0\})\over n};\\
{\mathbf {M}}^n_\Phi(\xi^+,\xi^-) &=1-{\mu^-_n(\{0\})\over n}\cdot\label{eq:ratio2bis}
\end{align}

It is easily seen from the very algorithm of Section \ref{sec:explo}, that the process 
$\procn{\tilde G_t}$ is Markov, but that the process $\procn{(\tilde \mu_t^+,\tilde\mu^-_t)}$, alone, is not. Indeed, for any $t$ the evolution from the graph $\tilde G_t$ to the graph $\tilde G_{t+1}$ depends on the whole connectivity of $\tilde G_t$ (which node is connected to which one), not only on its degree distribution $\tilde \mu_t$. 

On the other hand, regarding the construction of Section \ref{sec:CM} we have the following,

\begin{proposition}
\label{prop:Markov}
The process $\procn{(\mu_t^+,\mu^-_t)}$ is Markov. 
\end{proposition}

\begin{proof} Fix $t\in\llbracket 0,n \rrbracket$. Between times $t$ and $t+1$, 
the measures $\mu^+_t$ and $\mu^-_t$ are updated as follows: 
\begin{itemize}
\item If the drawn node $J^+$ on the left-hand side has a positive availability (Steps 1a) and then 2a) in the above procedure), then we first delete from $\mu^+_t$ the atom corresponding to the newly matched node $J^+$, 
and from $\mu^-_t$, the one corresponding to its match $J^-$ on the right-hand side.   
Then, $\mu^+_{t+1}$ is obtained by updating the availabilities of the unmatched nodes $J^+_1,...,J^+_\ell$ connected to $J^-$. We obtain that 
     \begin{equation}
    \label{eq:dynmu1}
    \left\{\begin{array}{ll}
    \mu_{t+1}^+ &= \mu_t^+ - \delta_{a_t(J^+)} + \displaystyle\sum_{l=1}^\ell \left(\delta_{a_t\left(J^+_l\right) - E^+_l} - \delta_{a_t\left(J^+_l\right)}\right);\\
    \mu_{t+1}^- &= \mu_t^- - \delta_{a_t(J^-)} + \displaystyle\sum_{l=1,\,l \ne \sigma(m)}^k\left(\delta_{a_t\left(J^-_l\right) - E^-_{l}} - \delta_{a_t\left(J^-_{l}\right)}\right),
    \end{array}\right.
    \end{equation}
    where sums over empty sets are understood as null. 
\item Now, if $J^+$ has a zero availability (Steps 1b) and 2b) in the above procedure), then this node switches from the set $U_t$ to the set $I_{t+1}$, and so the corresponding atom is deleted from the measure $\mu^+_t$. Thus we get 
\begin{equation}
    \label{eq:dynmu2}
    \left\{\begin{array}{ll}
    \mu_{t+1}^+ &= \mu_t^+ - \delta_{a_t(J^+)}\,;\\\\
    \mu_{t+1}^- &= \mu_t^-\,.
    \end{array}\right.
    \end{equation}
\end{itemize}
    In the above recurrence equations, the choice of $J^+$ is a measurable function of $\mu^+_t$ and of draws that are independent of $\sigma\left((\mu^+_s,\mu^+_s),\,s\le t\right)$. Also, in (\ref{eq:dynmu1}), from the assumption that the matching criterion is local, the choice of $J^-$ is also 
a measurable function of the availabilities of the nodes of the right-hand side, i.e., of $\mu^-_t$, together with a draw that is independent of $\sigma\left((\mu^+_s,\mu^+_s),\,s\le t\right)$, see (\ref{eq:defphi}). All the same, from the very construction of the uniform pairing procedure, the values of $J^+_1,...,J^+_\ell$, $E^+_1,...E^+_\ell$, $J^-_l,\, l\in \llbracket 1,k \rrbracket\setminus\{\sigma(m)\}$ and $E^-_l,\, l\in \llbracket 1,k \rrbracket\setminus\{\sigma(m)\}$  are measurable functions of $\mu_t^+,\mu_t^-$, and draws that are independent of $\sigma\left((\mu^+_s,\mu^+_s),\,s\le t\right)$. Therefore, (\ref{eq:dynmu1}-\ref{eq:dynmu2}) defines a proper Markov transition.
\end{proof}

As we have just proven, the measure-valued representation is Markov in the second construction, 
not in the first one. 
However, as the following result demonstrates, there is a close connexion between the two processes,

\begin{theorem}
\label{thm:coupling}
Let $n$ be a positive integer, and $\tilde G$ be a balanced bipartite graph of size $2n$, and let 
$\mathbf d^+$ and $\mathbf d^-$ be the respective 
degree distributions of the left-hand and right-hand side.  
Let $\procn{(\tilde \mu^+_t,\tilde \mu^-_t)}$ and $\procn{(\mu^+_t,\,\mu^-_t)}$ be the two measure-valued stochastic processes, defined respectively by (\ref{eq:defmutilde}) and (\ref{eq:defmu}), having common initial values 
\begin{equation*} 
\tilde \mu_0^+ =\mu_0^+=\sum_{i=1}^n \delta_{d^+(i)}\quad \mbox{ and }\quad 
\tilde \mu_0^- =\mu_0^-=\sum_{i=1}^n \delta_{d^-(i)}.
\end{equation*}
Let $G$ be the resulting multigraph of the second construction. 
Then, for any $t\in\llbracket 0,n \rrbracket$ and any couple of measures $(\nu^+,\nu^-)$, we get that 
\[\pr{(\mu^+_t,\mu^-_t) = (\nu^+,\nu^-) \mid G = \tilde G} = \pr{(\tilde \mu^+_t,\tilde \mu^-_t) = (\nu^+,\nu^-)}.\]
\end{theorem}   

\begin{proof}
Suppose that $G = \tilde G$, that is, the second construction eventually produces the graph 
$\tilde G$. We index the nodes of $V$ consistently in the two constructions.The initial couples of measures $(\tilde \mu_0^+,\tilde \mu_0^-)$ and $(\mu_0^+,\mu_0^-)$ coincide by assumption, and we then proceed by induction on $t$. 

Suppose that, at some time $t\in \llbracket 0,n-1 \rrbracket$ 
we have 
\begin{equation}
\label{eq:HR}
\tilde U_t=U_t\quad\mbox{ and }\quad a_t(v)=d_t(v)\,\mbox{ for all }\, v\mbox{ in } \tilde U_t=U_t.
\end{equation}
We will construct a coupling such that (\ref{eq:HR}) holds also at time $t+1$.  
For this, first, as $\tilde U_t=U_t$, at steps $\tilde 1$ and step $1$ respectively, we can set a common realization of a 
uniform draw on $\llbracket 1,|U^+_t| \rrbracket$, leading to the same values for 
$\tilde J^+$ and $J^+$. Then, 
\begin{itemize}
\item If $a_t(J^+)>0$, at step 1a),  as $G= \tilde G$ 
the uniform pairing procedure leads to the same set of neighbors for $J^+$ as in $\tilde G$, namely 
$J^-_1=\tilde J^-_1,...,J^-_k=\tilde J^-_k$. Then, as $d_t(\tilde J^+)=a_t(J^+)>0$, at steps $\tilde 2$a) and 2a) respectively, we can set the same value for 
$\tilde J^-$ and $J^-$ in $\tilde U^-_t=\tilde U^-_t$, using the common rule $\Phi$ and the the same uniform draw for the permutation $\sigma$. Finally, as $G= \tilde G$ the uniform pairing procedure leads to the same set of neighbors for $J^-$ (aside from $J^+$) as in $\tilde G$, namely 
$J^+_1=\tilde J^+_1,...,J^+_k=\tilde J^+_\ell$. We then obtain that 
\[\tilde U_{t+1}=\tilde U_t\setminus\{\tilde J^+,\tilde J^-\}=U_t\setminus\{J^+,J^-\}=U_{t+1}.\]
Second, as $G$ is a graph we get that $k=d_t(\tilde J^+)=a_t(J^+)$ and $\ell=d_t(\tilde J^-)=a_t(J^-)$. Moreover, for any 
$l\in\llbracket 1,k \rrbracket\setminus\{\sigma(m)\}$, at step 1a) we obtain that 
\[a_{t+1}(J^-_l)=a_t(J^-_l)-E^-_l=a_t(J^-_l)-1= d_t(J^-_l)-1 = a_{t+1}(J^-_l).\]
All the same, 
for any 
$l\in\llbracket 1,\ell \rrbracket\setminus\{\sigma(m)\}$, at step 2a) we obtain that 
\[a_{t+1}(J^+_l)=a_t(J^+_l)-E^+_l=a_t(J^+_l)-1= d_t(J^+_l)-1 = a_{t+1}(J^+_l).\]
Last, for any $v\in U_{t+1} \setminus \{J^+_1,...,J^+_\ell,J^-_1,...,J^-_k\}$ we get that 
\[a_{t+1}(v)=a_t(v)=d_t(v)=d_{t+1}(v).\]
\item If $a_t(J^+)=0$, then we obtain that 
\[\tilde U_{t+1}=\tilde U_t\setminus\{\tilde J^+\}=U_t\setminus\{J^+\}=U_{t+1}\]
and for all $v\in \tilde U_{t+1}=U_{t+1}$, 
\[a_{t+1}(v)=a_t(v)=d_t(v)=d_{t+1}(v).\]
\end{itemize}
In all cases, (\ref{eq:HR}) holds at step $t+1$. 

To conclude, we can then construct a coupling such that (\ref{eq:HR}) holds for all 
$t\in\llbracket 0,n \rrbracket$. From the very definitions (\ref{eq:defmutilde}) and (\ref{eq:defmu}), 
this implies in particular that $(\tilde \mu_{t}^+,\tilde \mu_{t}^-)=(\mu_{t}^+,\mu_{t}^-)$ for all 
$t\in\llbracket 0,n \rrbracket$. This concludes the proof. \end{proof}

We are now in a position to estimate the matching coverage of local algorithms on a given bipartite graph, 
as the size $2n$ of the latter grows to infinity, using an approximation of the Markov chain $\procn{(\mu_t^+,\mu^-_t)}$. 
Our strategy is as follows: first, using a procedure that is closely related to the
{Differential Equation Method} of \cite{WCours}, we can exploit the Markov property, Proposition \ref{prop:Markov}, 
to derive a large-graph approximation of the (suitably scaled) Markov chain 
$\procn{(\mu_t^+,\mu^-_t)}$ as the solution of an ODE, which we introduce in Section \ref{sec:fluid}.  
Then, we use the relations (\ref{eq:ratio1}) and (\ref{eq:ratio2}), together with 
our coupling result, Theorem \ref{thm:coupling}, to approximate the matching coverage for a given graph $G$ and a given matching criterion $\Phi$, by an explicit function 
of the measure-valued solution of the ODE under consideration - see (\ref{Mconst}) below. 
As examples, we then study the two particular cases of local algorithms, 
{\sc greedy} and {\sc minres}, introduced in Definition \ref{def:algo}.

\section{Hydrodynamic approximation}
\label{sec:fluid}
In this section, we introduce our hydrodynamic approximation of the construction of Section \ref{sec:CM}, namely, a system of measure-valued ODE that we will use as an approximation of the behavior of the process 
$\procn{(\mu_t^+,\mu^-_t)}$.

\begin{definition}[Hydrodynamic approximation]
Fix a local matching algorithm $\Phi$. 
Let for all $\varphi:\R\to\R$, $\mu^+,\,\mu^-\in\M(\N)$, $k\in\N$ and $t\ge 0$, 
\[
F^\Phi_{\varphi,\mu^-}(k)=\esp{\varphi\left(a_t(J^-)-1\right)\,|\,a_t(J^+)=k,\,(\mu^+_t,\mu^-_t)=(\mu^+,\mu^-)}\mathbf 1_{\N^*}(k). 
\]
A {hydrodynamic approximation} of the algorithm is a solution 
$\procu{(\barmul,\barmur)}$ in $\mathrm D\left([0,1],(\M(\N))^2\right)$, of the following family of systems of ODE's: 
for all $s\in[0,1]$, 
\begin{equation}
\label{eq:hydro}
\left\{\begin{array}{ll}
\displaystyle{\d\cro{\barmul_s,\varphi} \over \d s} &= \displaystyle-{\cro{\barmul_s,\varphi} \over \cro{\barmul_s,\mathbf 1}}
-{\cro{\barmul_s,F^\Phi_{X,\barmur_s}} \over \cro{\barmul_s,\mathbf 1}}{\cro{\barmul_s,X\Delta\varphi} \over \cro{\barmul_s,X}} \\
\displaystyle{\d\cro{\barmur_s,\varphi} \over \d s} &= \displaystyle-{\cro{\barmul_s,F^\Phi_{\varphi,\barmur_s}} \over \cro{\barmul_s,\mathbf 1}}
-{\cro{\barmul_s,X} \over \cro{\barmul_s,\mathbf 1}}{\cro{\barmur_s,X\Delta\varphi} \over \cro{\barmur_s,X}}
\end{array}\right.,\, \varphi \in \mathcal B_b(\N)\cup\{X,X^2\}.\end{equation}
\end{definition}

\subsection{Heuristics: a large-graph limit} 
\label{subsec:heuri}
Properly showing the convergence of the sequence of CTMC's $\procn{(\mu_t^+,\mu^-_t)}$, properly 
scaled, to a solution of (\ref{eq:hydro}), lies beyond the scope of this paper - and so does the proof of existence and uniqueness of the latter. 
However, let us provide heuristic arguments to intuitively justify the form of the right-hand side of (\ref{eq:hydro}) as an approximation of the dynamics of our construction. To do so, we study for all $\varphi:\R \to \R$, the drift of the process $\procn{\cro{\mu_t^{+},\varphi}}$, the drift of $\procn{\cro{\mu_t^{+},\varphi}}$ can be addressed similarly.
Fix $\varphi:\R \to \R$, two measures $\mu^+$ and $\mu^-$ and $t\in \llbracket 0,n-1 \rrbracket$. 
We address the terms on the right-hand side of (\ref{eq:dynmu1}) separately:  
\begin{itemize}
	\item First, as $J^+$ is drawned uniformly across among all nodes on the `+' side, we easily get that 
\begin{align}
\esp{\varphi\left(a_t(J^+)\right)\,|\,(\mu^+_t,\mu^-_t)=(\mu^+,\mu^-)} &= \sum_{k=0}^n \varphi(k)\pr{a_t(J^+)=k\,|\,(\mu^+_t,\mu^-_t)=(\mu^+,\mu^-)}\nonumber\\
& = \sum_{k=0}^n \varphi(k) \frac{\mu^+(k)}{\cro{\mu^+,\mathbf 1}}
= \frac{\cro{\mu^+,\varphi}}{\cro{\mu^+,\mathbf 1}}\cdot\label{eq:aptbis} \end{align}
\item Now, first observe that as the number of nodes grows large the probability of having multi-edges in the CM vanishes (see e.g. \cite{RDH}). In particular, with overwhelming probability, $J^-$ has $a_t(J^-)-1$ different neighbors which are not $J^+$, so we can make the approximation 
\begin{multline}\label{eq:apt0}
\esp{\sum_{l=1}^\ell\left(\varphi\left(a_t\left(J^+_l\right) - E^+_l\right) - \varphi\left(a_t\left(J^+_l\right)\right)\right)\,|\,(\mu^+_t,\mu^-_t)=(\mu^+,\mu^-)}\\ 
\simeq \esp{\sum_{l=1}^{a_t\left(J^-\right)-1}
\left(\varphi\left(a_t\left(J^+_l\right) - 1\right) - \varphi\left(a_t\left(J^+_l\right)\right)\right)\,|\,(\mu^+_t,\mu^-_t)=(\mu^+,\mu^-)}
.\end{multline}
Second, for a large $n$, independent sampling of half-edges {with} replacement are asymptotically equivalent to the sampling without replacement that are inherent to the procedure of uniform pairing of half-edges. Therefore, it follows from Wald's identity that 
\begin{multline}
\esp{\sum_{l=1}^{a_t\left(J^-\right)-1}\left(\varphi\left(a_t\left(J^+_l\right) - 1\right) - \varphi\left(a_t\left(J^+_l\right)\right)\right)\,|\,(\mu^+_t,\mu^-_t)=(\mu^+,\mu^-)}\\
\shoveleft{=-\esp{\sum_{l=1}^{a_t\left(J^-\right)-1}\Delta\varphi\left(a_t\left(J^+_l\right)\right) \,|\,(\mu^+_t,\mu^-_t)=(\mu^+,\mu^-)}}\\
\simeq -\esp{\Delta\varphi\left(a_t\left(J^+_1\right)\right) \,|\,(\mu^+_t,\mu^-_t)=(\mu^+,\mu^-)}
\esp{a_t\left(J^-\right)-1\,|\,(\mu^+_t,\mu^-_t)=(\mu^+,\mu^-)}. \label{eq:apt1}\end{multline}
But as the choice of half-edges are uniform, 
$a_t(J_1^+)$ follows the size-biased distribution relative to $\mu^+$, so we can write that 
\begin{align}
\esp{\Delta\varphi\left(a_t\left(J^+_1\right)\right)\,|\,\mu^+_t=\mu^+, \,\mu^-_t=\mu^-} 
&= \sum_{a=1}^n \Delta\varphi(a) \pr{a_t\left(J^+_1\right)=a\,|\,\mu^+_t=\mu^+, \,\mu^-_t=\mu^-}\nonumber\\
&= \sum_{a=1}^n \Delta\varphi(a) \frac{k\mu^+(k)}{\cro{\mu^+,X}}={\cro{\mu^+,X\Delta\varphi} \over \cro{\mu^+,X}}\cdot \label{eq:apt2}\end{align}


Also, in view of the above remark we obtain that 
\begin{multline*}
\esp{a_t\left(J^-\right)-1\,|\,(\mu^+_t,\mu^-_t)=(\mu^+,\mu^-)}\\
\begin{aligned}
&=\sum_{k=0}^n \esp{a_t\left(J^-\right)-1\,|\,a_t(J^+)=k,\,(\mu^+_t,\mu^-_t)=(\mu^+,\mu^-)}\pr{a_t(J^+)=k\,|\,(\mu^+_t,\mu^-_t)=(\mu^+,\mu^-)}\\
&=\sum_{k=0}^nF^\Phi_{X,\mu^-}(k)\pr{a_t(J^+)=k\,|\,(\mu^+_t,\mu^-_t)=(\mu^+,\mu^-)}\\
&= \sum_{k=0}^nF^\Phi_{X,\mu^-}(k){\mu^+(k) \over \cro{\mu^+,\mathbf 1}}= \frac{\cro{\mu^+,F^\Phi_{X,\mu^-}}}{\cro{\mu^+,\mathbf 1}}\cdot
\end{aligned}
\end{multline*}
 Gathering this together with (\ref{eq:apt1}-\ref{eq:apt2}) into (\ref{eq:apt0}), we obtain that 
\begin{equation*}
\esp{\sum_{l=1}^\ell\left(\varphi\left(a_t\left(J^+_l\right) - E^+_l\right) - \varphi\left(a_t\left(J^+_l\right)\right)\right)\,|\,(\mu^+_t,\mu^-_t)=(\mu^+,\mu^-)}
\simeq -\frac{\cro{\mu^+,F^\Phi_{X,\mu^-}}}{\cro{\mu^+,\mathbf 1}} {\cro{\mu^+,X\Delta\varphi} \over \cro{\mu^+,X}}\cdot\end{equation*}
\end{itemize}
This, together with (\ref{eq:aptbis}) in (\ref{eq:dynmu1}), yields to the following drift approximation: 
\begin{equation}
\label{eq:apt}
\esp{\cro{\mu_{t+1}^{+},\varphi} - \cro{\mu_{t}^{+},\varphi}\,|\,(\mu^{+}_t,\mu^{-}_t)=(\mu^+,\mu^-)}
\simeq -\frac{\cro{\mu^+,\varphi}}{\cro{\mu^+,\mathbf 1}} - \frac{\cro{\mu^+,F^\Phi_{X,\mu^-}}}{\cro{\mu^+,\mathbf 1}} {\cro{\mu^+,X\Delta\varphi} \over \cro{\mu^+,X}}\cdot 
\end{equation}

\paragraph{Scaling.} Let us first emphasize the dependence in $n$ of the various parameters for a given graph of size $2n$, by adding a superscript $n$ to all variables involved. In particular, we denote respectively by $\mu_t^{n,+}$ and $\mu_t^{n,-}$, the two measures at time $t$. 
Let us denote for any $(\mu^+,\mu^-)$ and any $G:\M(\N)^2\to \R$ such that the following expectations exist, the drift 
\begin{equation*}
{\mathscr Q}^nG\left(\mu^+,\mu^-\right) :=  \esp{G\left(\mu_{t+1}^{n,+},\mu_{t+1}^{n,+}\right) - G\left(\mu_{t}^{n,+},\mu_{t}^{n,+}\right)\,|\,(\mu^{n,+}_t,\mu^{n,-}_t)=(\mu^+,\mu^-)},\quad t \ge 0.
\end{equation*}
Then, defining for any $\varphi\in\mathcal B_b(\N)\cup\{X,X^2\}$, the mapping 
\[\Pi^1_\varphi:\begin{cases}
\M(\N)^2 &\longrightarrow \R\\
(\mu^+,\mu^-) &\longmapsto \cro{\mu^+,\varphi}
\end{cases}\,,\]
(\ref{eq:apt}) can be rewritten as 
\begin{equation}\label{eq:aptter}
\mathscr Q^n\Pi^1_\varphi\left(\mu^+,\mu^-\right)
\simeq \mathscr Q\Pi^1_\varphi\left(\mu^+,\mu^-\right)
:=-\frac{\cro{\mu^+,\varphi}}{\cro{\mu^+,\mathbf 1}} - \frac{\cro{\mu^+,F^\Phi_{X,\mu^-}}}{\cro{\mu^+,\mathbf 1}} {\cro{\mu^+,X\Delta\varphi} \over \cro{\mu^+,X}}\cdot
\end{equation}
Let us define the corresponding scaled, and interpolated, continuous-time processes: for all $t\in[0,1]$, we let 
\begin{equation}
\barmunl_t={1\over n}\mu^{n,+}_{\left\lfloor nt \right\rfloor}\quad\mbox{ and }\quad \barmunr_t={1\over n}\mu^{n,-}_{\left\lfloor nt \right\rfloor}.
\label{eq:norm}
\end{equation}
This scaling corresponds to an acceleration in time of factor $n$ (each consecutive steps being $1/n$ units of time apart), compensated by a scaling in space, of the same magnitude, the weight of each atom being divided by $n$. 
As the approximate drift in the right-hand side of (\ref{eq:aptter}) is independent of $n$ and bilinear in $(\mu^+,\mu^-)$, using martingale representations, after scaling and taking $n$ to infinity, it is then standard (again, see the Differential equation method in \cite{WCours}), that any limiting process 
$\procu{\barmu_t}:=\procu{(\barmu^+_t,\barmu^-_t)}$ of $\procu{\barmun_t}:=\procu{(\barmunl_t,\barmunr_t)}$ 
solves the integral equation
\begin{equation*}
\cro{\barmul_t,\varphi}-\cro{\barmul_0,\varphi} = \Pi^1_\varphi(\barmu_t)-\Pi^1_\varphi(\barmu_0)
=\int_0^t \mathscr Q \Pi^1_\varphi\left(\barmu_s\right)\d s,\quad t\in[0,1],\end{equation*}
if a solution does exist. 
Plugging (\ref{eq:aptter}) in the above, and differentiating in $t$, yields to the first equation of (\ref{eq:hydro}) at all $t$ and for all $\varphi$. Retrieving the second equation of (\ref{eq:hydro}) can be done in a similar fashion. 


\subsection{Hydrodynamic approximation for particular local algorithms}
We now make precise the form of the systems of ODEs (\ref{eq:hydro}) for the two particular local algorithms introduced 
in Definition \ref{def:algo}. We address successively $\Phi=\textsc{Greedy}$ and $\Phi=\textsc{Minres}$.

\subsubsection{$\Phi=\textsc{Greedy}$.} For any $(\mu^+,\mu^-)\in\M(\N)^2$, $t\in [0,1]$ and any $k\in\N^*$, given that 
$\mu^+_t=\mu^+$, $\mu^-_t=\mu^-$ and that the degree of $J^+$ is $k>0$, by the very definition of the uniform pairing procedure, the match $J^-$ of $J^+$ is just determined by a uniform draw of one of its half-edge. Thus its availability $a_t(J^-)$ at $t$ follows the 
size-biased distribution associated to $\mu^-$, namely we get that for all $a\in\N$, 
\[\pr{a_t(J^-)=a \,|\, a_t(J^+)=k,\,(\mu^+_t,\mu^-_t)=(\mu^+,\mu^-)}={a\mu^-(a) \over \cro{\mu^-,X}}\cdot\]
Consequently, we get that for all $k\in\N$ and $\varphi\in\mathcal B_b(\N)\cup\{X,X^2\}$, 
\begin{align*}
F^{\scriptsize{\textsc{greedy}}}_{\varphi,\mu^-}(k)&=\esp{\varphi\left(a_t(J^-)-1\right)\,|\,K^+=k,\,\mu^-_t=\mu^-}\\
								       &= \sum_{a\in\N} \varphi(a-1){a\mu^-(a) \over \cro{\mu^-,X}}\mathbf 1_{\N^*}(k)
								       = {\cro{\mu^-,X\tau_1\varphi} \over \cro{\mu^-,X}}\mathbf 1_{\N^*}(k),
								       \end{align*}
								       where we denote $\tau_1\varphi(x)=\varphi(x-1)$ for all $x$. 
Plugging this into (\ref{eq:hydro}) we obtain that for all $t\le 1$, 
\begin{equation}
\begin{cases}
\displaystyle{\d\cro{\barmul_s,\varphi} \over \d s} &= \displaystyle-{\cro{\barmul_s,\varphi} \over \cro{\barmul_s,\mathbf 1}}
-{\cro{\barmul_s,\mathbf 1_{\N^*}} \over \cro{\barmul_s,\mathbf 1}}{\cro{\barmur_s,X^2-X} \over \cro{\barmur_s,X}}{\cro{\barmul_s,X\Delta\varphi} \over \cro{\barmul_s,X}};\\\\
\displaystyle{\d\cro{\barmur_s,\varphi} \over \d s} &=\displaystyle -{\cro{\barmur_s,X\tau_1\varphi} \over \cro{\barmur_s,X}}{\cro{\barmul_s,\mathbf 1_{\N^*}} \over \cro{\barmul_s,\mathbf 1}}
-{\cro{\barmul_s,X} \over \cro{\barmul_s,\mathbf 1}}{\cro{\barmur_s,X\Delta\varphi} \over \cro{\barmur_s,X}}\\\\
&= \displaystyle-{\cro{\barmur_s,X\varphi} \over \cro{\barmur_s,X}}{\cro{\barmul_s,\mathbf 1_{\N^*}} \over \cro{\barmul_s,\mathbf 1}}
-{\cro{\barmul_s,X-\mathbf 1_{\N^*}} \over \cro{\barmul_s,\mathbf 1}}{\cro{\barmur_s,X\Delta\varphi} \over \cro{\barmur_s,X}}\cdot \label{eq:hydroG}
\end{cases}
\end{equation}

\subsubsection{$\Phi=\textsc{Minres}$.} 
By the very definition of {\sc Minres}, for all $(\mu^+,\mu^-)\in\M(\N)^2$, $t\in [0,1]$ and all $k\in\N^*$, conditional on $a_t(J^+)=k$ and $(\mu^+_t,\mu^-_t)=(\mu^+,\mu^-)$, for all all $a\in \N^*$ $a_t(J^-)\ge a$ means that 
in the uniform pairing procedure, none of the neighbors of $J^+$ are of availability strictly less than $a$. 
Thus, as in the large-graph limit, draws without replacement can be approximated by draws {with} replacement, and at each step, all the neighbors of $J^+$ have with overwhelming probability, a single neighbor on the right-hand side, we get tho the approximation  
\[\pr{a_t(J^-) \ge a \,|\, a_t(J^+)=k,\,(\mu^+_t,\mu^-_t)=(\mu^+,\mu^-)}\simeq \left(\frac{\cro{\mu^-,X\mathbf 1_{[a,+\infty)}}}{\cro{\mu^-,X}}\right)^k.\]
Therefore, we obtain that 
\begin{align*}
F^{\scriptsize{\textsc{minres}}}_{\varphi,\mu}(k)= &\esp{\varphi\left(a_t(J^-)-1\right)\,|\, a_t(J^+)=k,\,(\mu^+_t,\mu^-_t)=(\mu^+,\mu^-)}\\
=&\sum_{a\in\N^*} \varphi(a-1)\Biggl(\pr{a_t(J^-) \ge a \,|\, a_t(J^+)=k,\,(\mu^+_t,\mu^-_t)=(\mu^+,\mu^-)}\\
 &\phantom{\sum_{a\in\N} \varphi(a-1)\Biggl(}\quad - \pr{a_t(J^-) \ge a+1 \,|\, a_t(J^+)=k,\,(\mu^+_t,\mu^-_t)=(\mu^+,\mu^-)}\biggl)\\
\simeq&\sum_{a\in\N^*} \varphi(a-1)\left(\frac{\cro{\mu^-,X\mathbf 1_{[a,+\infty)}}^k - \cro{\mu^-,X\mathbf 1_{[a+1,+\infty)}}^k}{\cro{\mu^-,X}^k}\right).\end{align*}
Plugging this into (\ref{eq:hydro}) yields to
\begin{equation}
\begin{cases}
\displaystyle{\d\cro{\barmul_s,\varphi} \over \d s} \simeq&\displaystyle -{\cro{\barmul_s,\varphi} \over \cro{\barmul_s,\mathbf 1}}\\\\
&\displaystyle-{\cro{\barmul_s,X\Delta\varphi} \over \cro{\barmul_s,X}} \sum_{k\in\N} \sum_{a\in\N^*}\left(\frac{\cro{\barmur_s,X\mathbf 1_{[a,+\infty)}}^k - \cro{\barmur_s,X\mathbf 1_{[a+1,+\infty)}}^k}{\cro{\barmur_s,X}^k}\right){(a-1)\barmul_s(k) \over \cro{\barmul_s,\mathbf 1}};
\\\\
\displaystyle{\d\cro{\barmur_s,\varphi} \over \d s} \simeq& \displaystyle-\sum_{k\in\N} \sum_{a\in\N^*}\left(\frac{\cro{\barmur_s,X\mathbf 1_{[a,+\infty)}}^k 
- \cro{\barmur_s,X\mathbf 1_{[a+1,+\infty)}}^k}{\cro{\barmur_s,X}^k}\right){(a-1)\barmul_s(k) \over \cro{\barmul_t,\mathbf 1}}
\\\\
&\displaystyle-{\cro{\barmul_s,X} \over \cro{\barmul_s,\mathbf 1}}{\cro{\barmur_s,X\Delta\varphi} \over \cro{\barmur_s,X}}\cdot
\end{cases}\label{eq:hydroM}
\end{equation}

\subsection{Approximating the matching coverage}
We are now ready to introduce our estimate of the matching coverage in the large-graph limit. 
Recall, that in the construction of Section \ref{sec:explo}, the isolated nodes at $t=n$ 
on the `-' side are precisely the undetermined nodes having degree zero at the end of the construction. All the same, in the construction of Section \ref{sec:CM} the isolated nodes at $t=n$ on the `-' side are the undetermined nodes of the `-' side having degree zero at the final step, because at completion of the algorithm, all half-edges have been paired. 

Second, by construction there are in both cases, as many isolated nodes on both sides 
of the bipartition. Therefore, we get that 
\[|\tilde I_n| = 2\tilde \mu^-_n(\{0\}) \quad \mbox{ and }\quad |I_n| = 2\mu^-_n(\{0\}).\]
All in all, it follows from (\ref{eq:ratio1}) and (\ref{eq:ratio2}), that the two matching coverages are respectively given by 
\begin{align}
\label{eq:ratio1bis}
\tilde{\mathbf M}^n_\Phi(G)&=1-{\tilde \mu^-_n(\{0\})\over n};\\
{\mathbf {M}}^n_\Phi(\xi^+,\xi^-) &=1-{\mu^-_n(\{0\})\over n}=1-\barmunr_1(\{0\})\cdot\label{eq:ratio2bis}
\end{align}
Finally, in view of the large-graph approximation of Subsection \ref{subsec:heuri} we can approximate the matching coverage of the second construction by 
\begin{align}
\overline{\mathbf M}_\Phi(\xi^+,\xi^-) := 1 -  \barmur_1(\{0\}),
\label{Mconst}
\end{align}
where $\procu{\barmu_t}=\procu{(\barmu^+_t,\barmu^-_t)}$ is a solution to (\ref{eq:hydro}). 
In turn, from our coupling result, Theorem \ref{thm:coupling}, for a fixed graph $G$ the matching coverage 
after performing the first construction on $G$, given by (\ref{eq:ratio1bis}), can be approximated by (\ref{Mconst}) 
for the same initial degree distribution. 

\section{Simulations}
\label{sec:simu}
In this section, we illustrate and complete the above results by way of simulations. 
In Subsection \ref{subsec:3regular}, we test the empirical convergence of the matching coverage of the construction 
of Section \ref{sec:explo}, given by (\ref{eq:ratio1bis}), under both algorithms {\sc Greedy} and {\sc Minres}, to the approximate matching size predicted by the ODEs through formula (\ref{Mconst}), for $G$ a $3$-regular graph. 
Then, in Subsection \ref{subsec:distrib} we address, first, a comparison of the performances of the two algorithms, 
and second, the influence of the parameters on the matching size, along various degree distributions. 
 Finally, in Subsection \ref{subsec:topo} we study the influence of the topology of the graph under consideration, 
 by comparing the results when performing the algorithm of Section \ref{sec:explo}, and the algorithm of 
 Section \ref{sec:CM}. 
 We conclude with a comparison to the results in \cite{Karp}.

\subsection{Convergence of the Matching Size : a case study on 3-regular graphs}
\label{subsec:3regular}

In this section we use 3-regular graphs to the test the convergence of $\tilde{\mathbf M}^n_\Phi(G)$ to $\overline{\mathbf M}_\Phi(\delta_3,\delta_3)$, for $G$ a $3$-regular graph and the initial degree distribution $\mu_0^+ =\mu_0^- = n\delta_3$, 
for $n$ the number of vertices on both sides. In this case we readily obtain that 
\begin{equation}
\barmul_0 = \delta_3\quad\mbox{ and }\quad \barmur_0 = \delta_3.
\end{equation}
For {\sc Greedy}, specializing the system (\ref{eq:hydroG}) successively to $\varphi=\mathbf 1_{\{0\}}$, 
$\varphi=\mathbf 1_{\{1\}}$, $\varphi=\mathbf 1_{\{2\}}$, $\varphi=\mathbf 1_{3}$, yields to 
%
\begin{equation}\label{eq:hydroG3regular}
\begin{cases}
\displaystyle\frac{\d\barmul_s(0)}{\d s} & =  \displaystyle- \left( \frac{\barmul_s(0)}{m(\barmul_s)} + \frac{m^*(\barmul_s)}{m(\barmul_s)}\left(\frac{Q(\barmur_s)}{E(\barmur_s)} - 1\right)\frac{- \barmul_s(1)}{E(\barmul_s)} \right); \\
\displaystyle\frac{\d\barmul_s(1)}{\d s} &=  \displaystyle- \left( \frac{\barmul_s(1)}{m(\barmul_s)} + \frac{m^*(\barmul_s)}{m(\barmul_s)}\left(\frac{Q(\barmur_s)}{E(\barmur_s)} - 1\right)\frac{\barmul_s(1) - 2\barmul_s(2)}{E(\barmul_s)} \right); \\
\displaystyle\frac{\d\barmul_s(2)}{\d s} &=  \displaystyle- \left( \frac{\barmul_s(2)}{m(\barmul_s)} + \frac{m^*(\barmul_s)}{m(\barmul_s)}\left(\frac{Q(\barmur_s)}{E(\barmur_s)} - 1\right)\frac{2\barmul_s(2) - 3\barmul_s(3)}{E(\barmul_s)} \right); 
\\
\displaystyle\frac{\d\barmul_s(3)}{\d s} &= \displaystyle - \left( \frac{\barmul_s(3)}{m(\barmul_s)} + \frac{m^*(\barmul_s)}{m(\barmul_s)}\left(\frac{Q(\barmur_s)}{E(\barmur_s)} - 1\right)\frac{3\barmul_s(3)}{E(\barmul_s)} \right); \\[3mm]
\displaystyle\frac{\d\barmur_s(0)}{\d s} & =  \displaystyle- \left(\frac{E(\barmul_s)}{m(\barmul_s)} - \frac{m^*(\barmul_s)}{m(\barmul_s)}\right)\frac{- \barmur_s(1)}{E(\barmur_s)};  \\
\displaystyle\frac{\d\barmur_s(1)}{\d s} & = \displaystyle - \left( \frac{\barmur_s(1)}{E(\barmur_s)}\frac{m^*(\barmul_s)}{m(\barmul_s)} + \left(\frac{E(\barmul_s)}{m(\barmul_s)} - \frac{m^*(\barmul_s)}{m(\barmul_s)}\right)\frac{\barmur_s(1) - 2\barmur_s(2)}{E(\barmur_s)} \right); \\
\displaystyle\frac{\d\barmur_s(2)}{\d s} & =  \displaystyle-\left( \frac{2\barmur_s(2)}{E(\barmur_s)}\frac{m^*(\barmul_s)}{m(\barmul_s)} + \left(\frac{E(\barmul_s)}{m(\barmul_s)} - \frac{m^*(\barmul_s)}{m(\barmul_s)}\right)\frac{2\barmur_s(2) - 3\barmur_s(3)}{E(\barmur_s)} \right);\\
\displaystyle\frac{\d\barmur_s(3)}{\d s} & =  -\displaystyle \left( \frac{3\barmur_s(3)}{E(\barmur_s)}\frac{m^*(\barmul_s)}{m(\barmul_s)} + \left(\frac{E(\barmul_s)}{m(\barmul_s)} - \frac{m^*(\barmul_s)}{m(\barmul_s)}\right)\frac{ 3\barmur_s(3)}{E(\barmur_s)} \right),\quad s \in [0,1], 
\end{cases}
\end{equation}
where we denote for all $\mu\in\M(\N)$, 
\begin{equation*}\begin{cases}
m(\mu) &= \cro{\mu,\mathbf 1} = \mu(0) +\mu(1) + \mu(2) + \mu(3);\\
m^*(\mu) & =  \mu(1) + \mu(2) + \mu(3); \\
E(\mu) &= \cro{\mu,X} = \mu(1) + 2\mu(2) + 3\mu(3);\\
Q(\mu) &= \cro{\mu,X^2} = \mu(1) + 4\mu(2) + 9\mu(3).
\end{cases}
\end{equation*}

\begin{figure}[H]
\begin{subfigure}{.5\textwidth}
  \centering

	\includegraphics[width=1\linewidth]{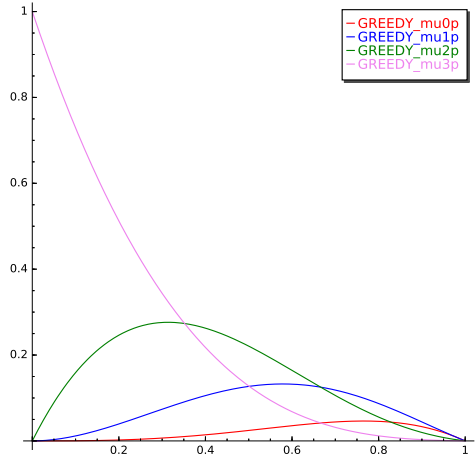}
  \caption{$\barmu^{+}_s(k)$ for $k=0,\dots,3$ and $s\in [0,1]$.}
  \label{fig:greedy3p}
\end{subfigure}
\begin{subfigure}{.5\textwidth}
  \centering

\includegraphics[width=1\linewidth]{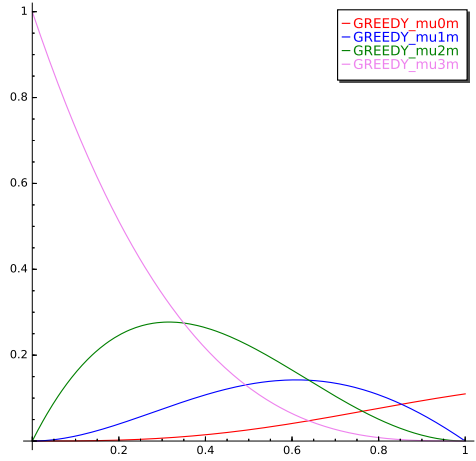}
  \caption{$\barmu^{-}_s(k)$ for $k=0,\dots,3$ and $s\in [0,1]$.}
 \label{fig:greedy3m}
\end{subfigure}

\medskip

\caption{{\sc Greedy}-EDO results for $\barmu^{+}_0=\barmu^{-}_0 = \delta_3$}
\label{fig:greedy3}
\end{figure}

A numerical resolution of the system of ODEs (\ref{eq:hydroG3regular}) is presented in Figure \ref{fig:greedy3}. 
We obtain the final value 
$\barmur_1(0) = 0.1098$ (which is the final value of the red curve in the right-hand curve of Figure \ref{fig:greedy3}). 
From this, we deduce the approximate matching coverage
\begin{equation}\label{eq:ratioG3regular}
\overline{\mathbf M}_{\textsc{Greedy}}(\delta_3,\delta_3) = 1 - \barmur_1(0) = 0.8902.
\end{equation}

Regarding $\textsc{Minres}$, similarly to (\ref{eq:hydroG3regular}) we can obtain a system of ODEs that specializes (\ref{eq:hydroM}) to the case of the $3$-regular degree distribution, which we skip for brevity. From this, we deduce the approximate matching coverage 
\begin{equation}\label{eq:ratioM3regular}
\overline{\mathbf M}_{\textsc{Minres}}(\delta_3,\delta_3) = 1 - \barmur_1(0) = 0.9378.
\end{equation}


To illustrate the convergence of $\tilde{\mathbf M}^n_\Phi(G)$ to $\overline{\mathbf M}_\Phi(\delta_3,\delta_3)$ in both cases, we proceed as follows: for each value of $n$ from 10 to 10 000, we implement the bipartite configuration model 
to draw a $3$-regular graph as a realization of the CM. Then, we run each algorithm once and then plot the matching coverage.  
The following figures compile the evolution of the simulated matching coverages as the graph size $n$ grows large.                                                                                                                                                             
\begin{figure}[H]
\allowdisplaybreaks[1]
\centering
	
			\includegraphics[width=0.7\linewidth]{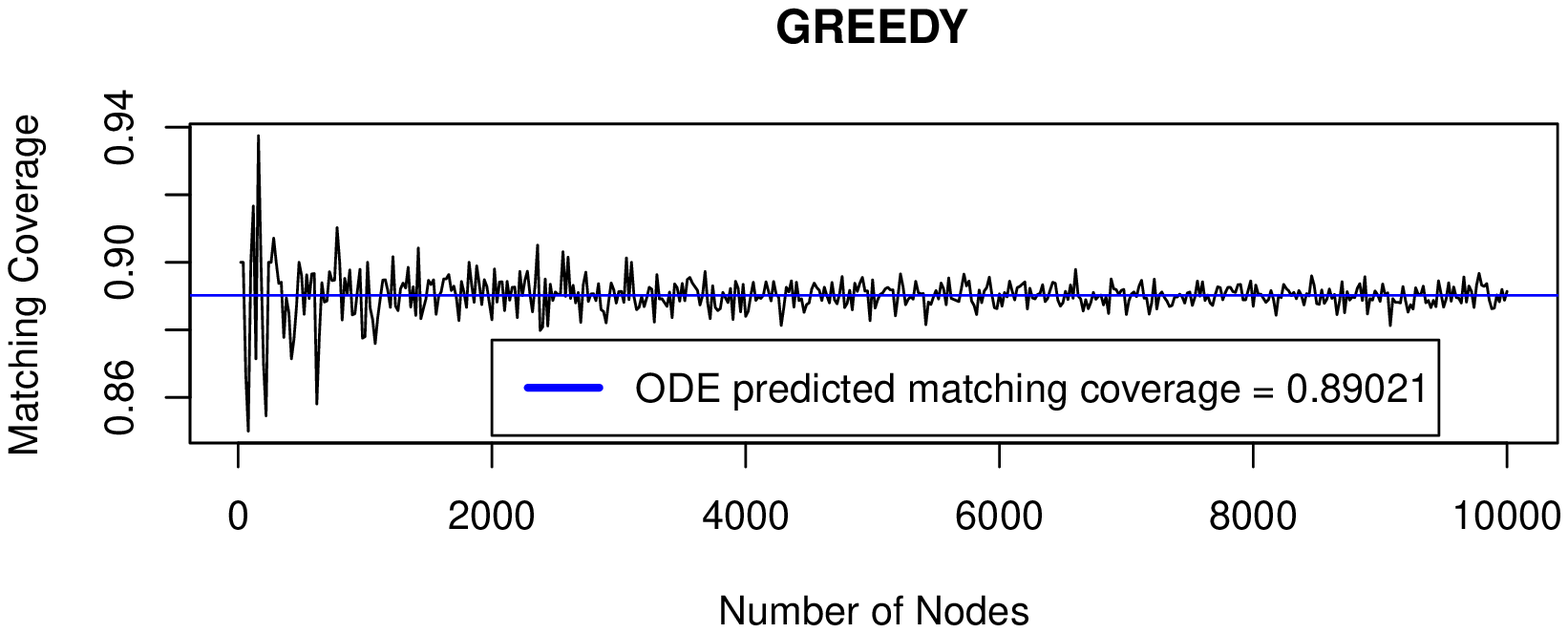}
			\caption{Matching coverage of {\sc Greedy} as the graph size tends to $\infty$}
\centering
	
			\includegraphics[width=0.7\linewidth]{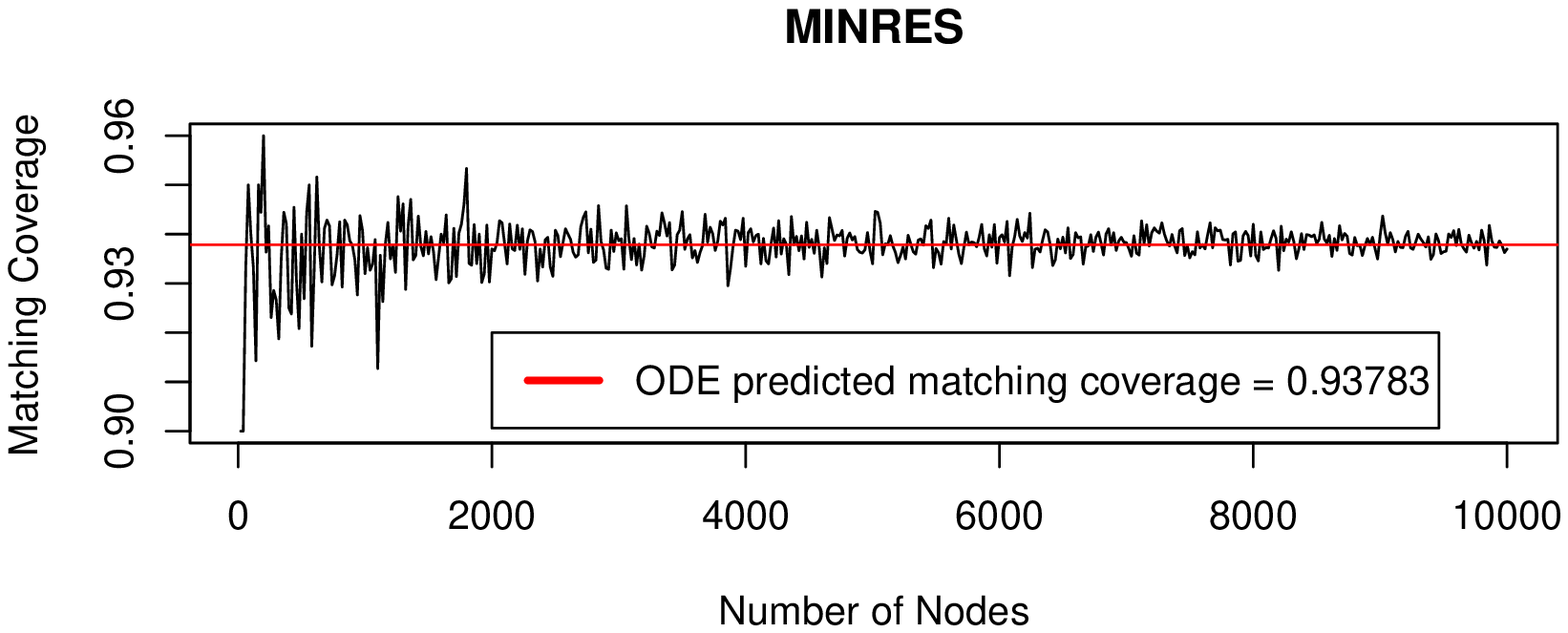}
			\caption{Matching coverage of {\sc Minres} as the graph size tends to $\infty$}
			\label{min3}
\end{figure}

As the graph gets larger, the fluctuations get smaller and smaller, heuristically showing the convergence of the matching coverage $\tilde{\mathbf M}^n_\Phi(G)$ to $\overline{\mathbf M}_\Phi(\delta_3,\delta_3)$, for both algorithms. To complete these results, for various graph sizes we ran $N=50$ iterations of the previous procedure. Means and standard deviations of the corresponding statistical distributions for $\tilde{\mathbf M}^n_\Phi(G)$, are given in Table \ref{tab:AvgConv}

\begin{table}[h!]
	\begin{center}
\begin{tabular}{|c|c|c|c|c|c|c||c|}
		\hline	
   &Graph Size $n$ & 200 & 500 & 1000 & 3000 & 5000 & $\overline{\mathbf M}_\Phi(\delta_3,\delta_3)$ \\ \hline	\hline
   $\tilde{\mathbf M}^n_{\textsc{Greedy}}(G)$ & Mean & 0.8904  & 0.8916  & 0.8911  & 0.8897   & 0.8898 & \textbf{0.8902} \\ \hline	
	& Std Dev &  0.0198  & 0.0109  & 0.009  & 0.0041   & 0.00311 &  \\ \hline
	$\tilde{\mathbf M}^n_{\textsc{Minres}}(G)$ & Mean & 0.9356  & 0.9365  & 0.9396  & 0.9378  & 0.9385 & \textbf{0.9378}   \\ \hline
	&Std Dev &  0.0148  & 0.0096  & 0.0052  & 0.0040   & 0.0025 &   \\ \hline
	
 \end{tabular}
\end{center}
	\caption{Recap. of Average Matching Size}
	\label{tab:AvgConv}
\end{table} 

Table \ref{tab:AvgConv} confirms the results that Figure \ref{min3} preluded. The shrinking of the standards deviation confirms the heuristic convergence to a deterministic value. It also stresses on the better performances of {\sc Minres} with respect to {\sc Greedy} for this particular degree distribution. In the next part, we develop this comparison for a 
      larger range of degree distributions.

%
%

\subsection{An Array of Degree Distributions}
\label{subsec:distrib}

After the illustration of the convergence to the solution of the system of ODEs in the previous section, in this section 
we study the evolution of the performances of the two algorithms {\sc Greedy} and {\sc Minres},  
along various degree distributions and various parameters. 

Our procedure is the following: for each degree distribution, we implement the bipartite CM to generate a large graph (of size $n=10^4$ nodes), 
in which the degrees of the nodes form a $n$-sample of the prescribed degree distribution, after testing the graphicality of the latter degree distribution (i.e., the feasibility of the generation of the graph). 
      For each distribution, we then ran $50$ iterations of both algorithms, following the first construction of Section \ref{sec:explo}. 

\paragraph{Poisson distributions.}
We first address the class of Poisson distributions, which are well know to be the asymptotic degree distributions  
of Erdös-Rényi graphs. 
Distributions of the matching coverage for the {\sc Greedy} (respectively, {\sc Minres}) algorithm are given in Figure \ref{fig:grepois} 
(resp., Figure \ref{fig:minpois}), for Poisson distributions of various parameters. 
For comparing the two algorithms, these distributions are gathered in Figure \ref{fig:totpois}.

\begin{figure}[H]
\begin{subfigure}{.5\textwidth}
  \centering

	\includegraphics[width=1\linewidth]{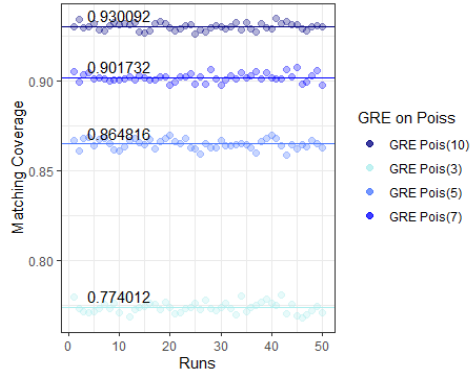}
  \caption{{\sc Greedy} on Poisson Laws}
  \label{fig:grepois}
\end{subfigure}
\begin{subfigure}{.5\textwidth}
  \centering

\includegraphics[width=1\linewidth]{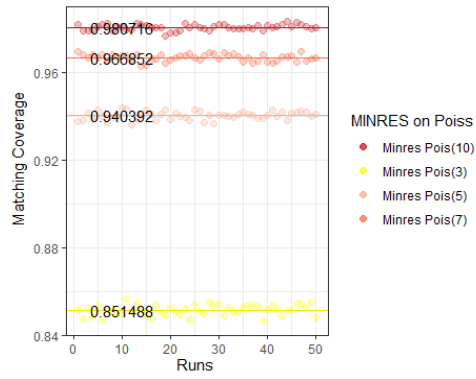}
  \caption{{\sc Minres} on Poisson Laws}
  \label{fig:minpois}
\end{subfigure}

\centering
\begin{subfigure}{1\textwidth}
  \centering
  
	\includegraphics[width=1\linewidth]{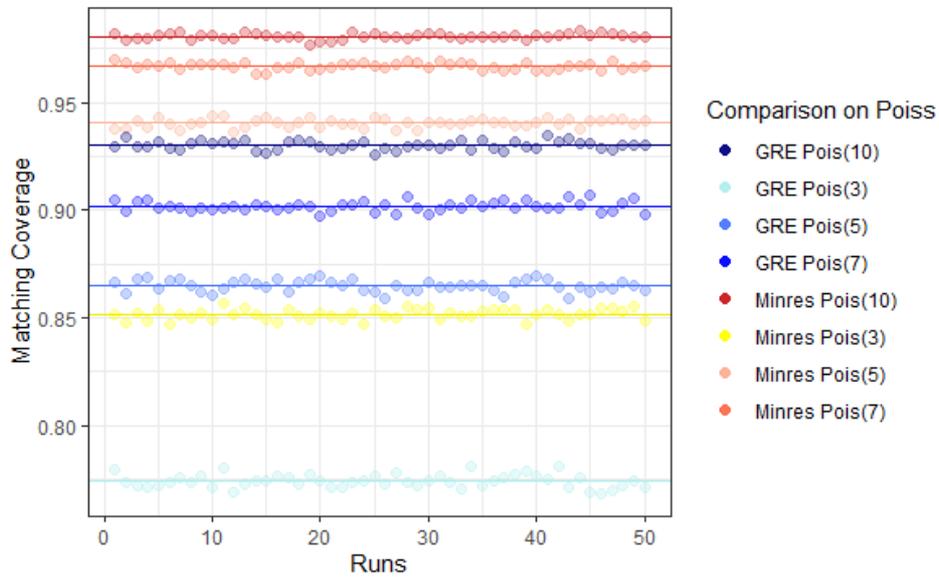}
  \caption{Comparison on Poisson Law}
  \label{fig:totpois}
\end{subfigure}

\caption{Performance on Poisson Degree Distribution}
\label{fig:pois}
\end{figure}

\paragraph{Regular bipartite graphs.}
We now address various degree distributions that correspond to regular bipartite graphs: each node has the same degree, 
in other words we have $\xi^+=\xi^-=\delta_p$, for some $p\in\N^*$, thereby generalizing the study of Sub-section \ref{subsec:3regular}. 
Notice that such bipartite graphs always admit a perfect matching. The distributions of matching coverage are 
given in Figures \ref{fig:grereg}, \ref{fig:minreg} and \ref{fig:totreg}.

\begin{figure}[H]
\begin{subfigure}{.5\textwidth}
  \centering
  
	\includegraphics[width=1\linewidth]{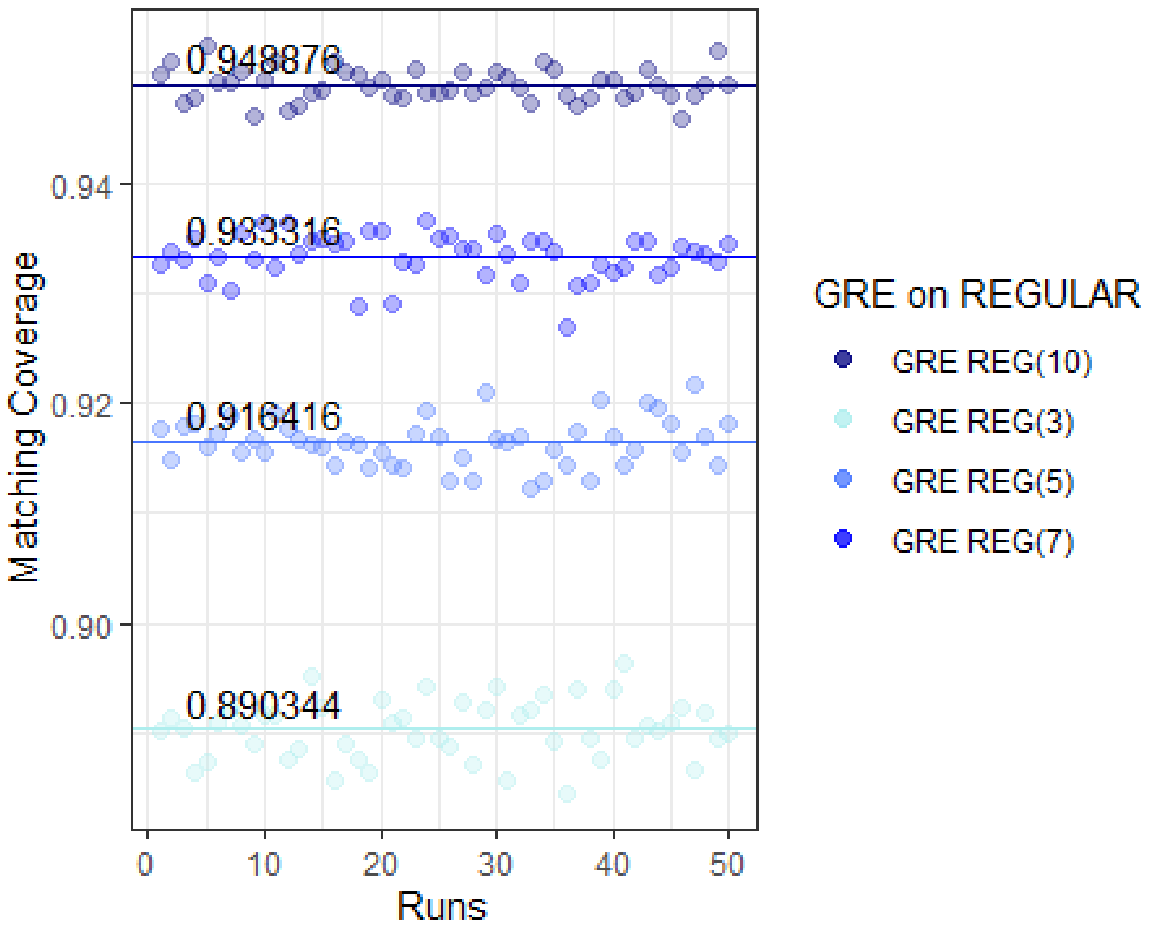}
  \caption{{\sc Greedy} on Regular Degree Distr.}
  \label{fig:grereg}
\end{subfigure}%
\begin{subfigure}{.5\textwidth}
  \centering
 
	\includegraphics[width=1\linewidth]{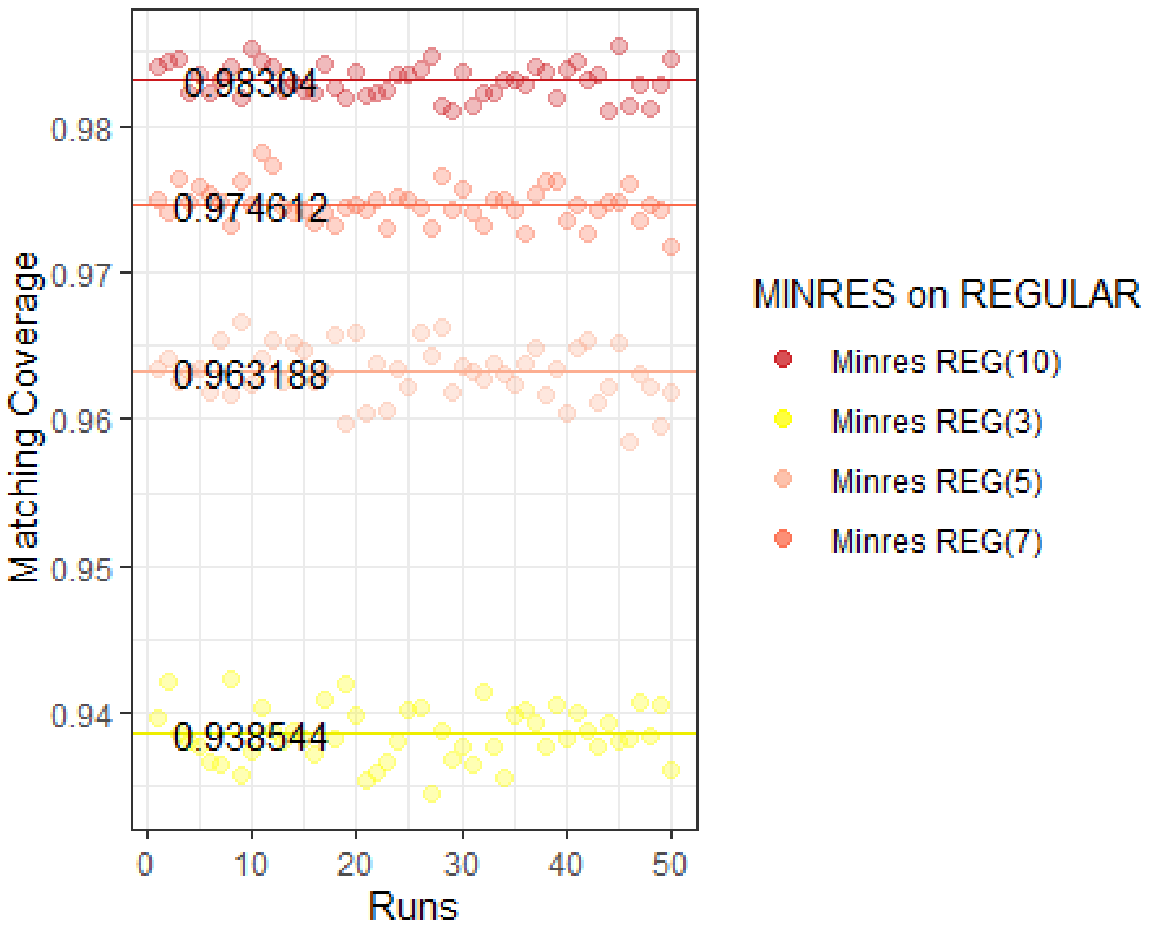}
  \caption{{\sc Minres} on Regular Degree Distr.}
  \label{fig:minreg}
\end{subfigure}
\centering
\begin{subfigure}{1\textwidth}  
\includegraphics[width=0.9\linewidth]{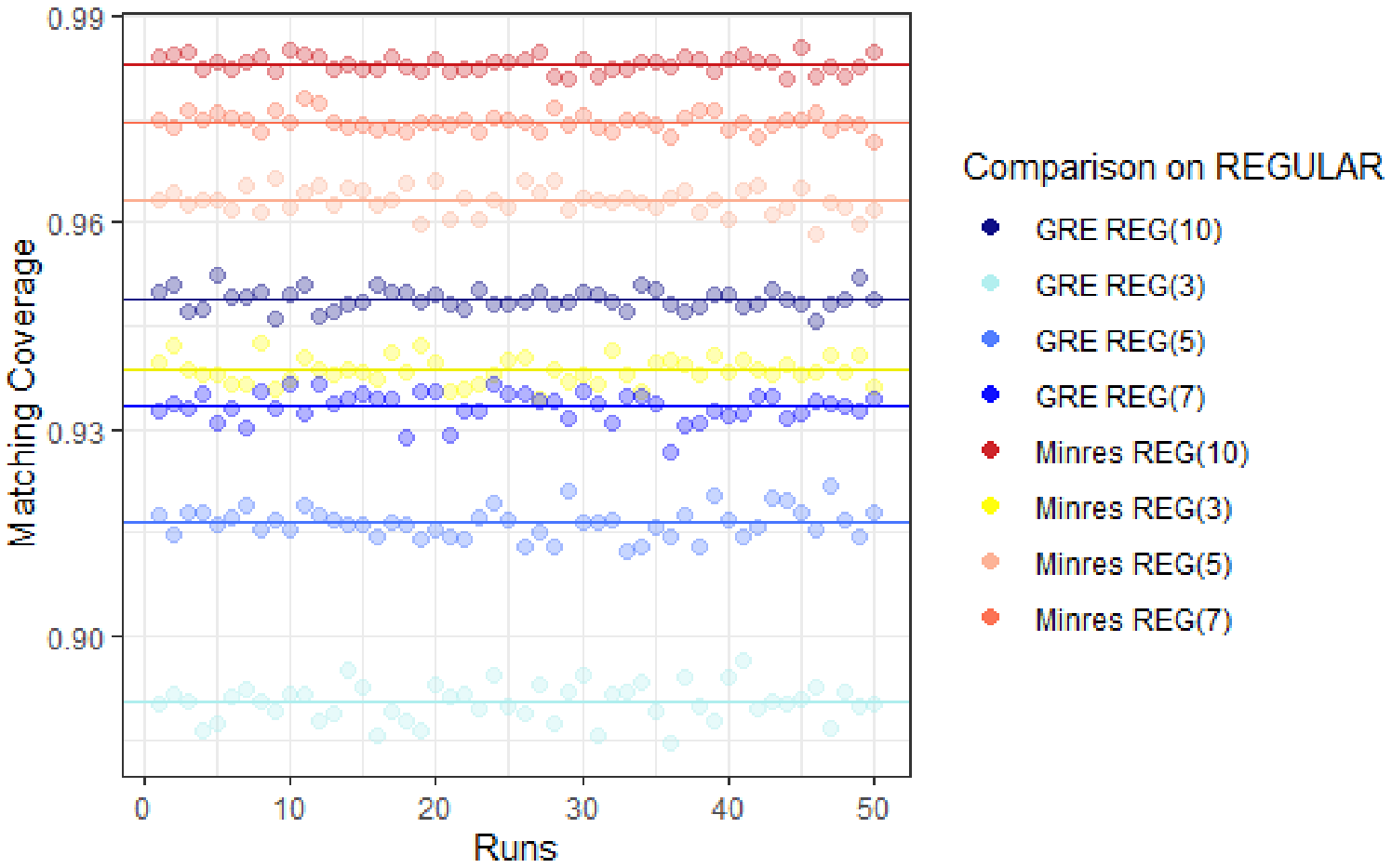}
  \caption{Comparison on Regular Distribution}
  \label{fig:totreg}
	
\end{subfigure}

\caption{Performance on Regular Degree Distr.}
\label{fig:reg}
\end{figure}

These results illustrate well the increase of the performance of both algorithms, as the degree increase. 
For each designated distribution, we can also confirm that the matching coverage of {\sc Minres} is consistently larger than that of 
{\sc Greedy}.  

Comparing the sub-figures of Figure \ref{fig:pois} with the corresponding sub-figures of Figure \ref{fig:reg}, for the same average degree, we make the two  following observations: 
First, both algorithms consistently perform better on regular graphs than on graphs with Poisson degree distributions, and the same mean. 
We conjecture that this phenomenon is due to the variance of Poisson degrees in the first panel:  by restricting \mbox{choice}s, the optimal partners for certain nodes (which we are trying to reach) might get blocked, while regular distributions provide more latitude to chose a match without {missing} an optimal one. Second, for both algorithms the distributions of matching coverage are more spread on regular graphs than on graphs having Poisson degrees. It seems that a uniform initial degree distribution provides more opportunities to deflect from \textit{typical} runs, while the variance of degrees more often restricts the \mbox{choice}s, creating a disparity between nodes.


\subsection{On the optimal Algorithm of Karp, Vazirani and Vazirani}
\label{subsec:topo}

In \cite{Karp}, Karp, Vazirani and Vazirani present a different approach for online bipartite matchings {on graphs}. The authors define online algorithms as a way of picking the matches of ``girls'' (i.e., nodes on the `+' side) that arrive one by one, based only on the identity of their neighbors on the `boys' side (i.e., the `-' side). By working with the adjacency matrix of the graph, it means that the columns are revealed one by one and the match is processed upon knowing the information of the current column. Our approach based on local algorithms, is a bit different, and actually use more information. Indeed, we also consider the degree of each neighbor of the incoming `girl', and the neighbors of its match of her 'boy' neighbor. 

Second, the performance metrics considered in \cite{Karp} is the so-called \textit{adversary approach}. It allows to get a lower bound on the performance of online algorithms, as defined above. Namely, Theorem 2 in \cite{Karp} states that in that context, the matching coverage is at least $1-1/e - {o(n)/n}$. This also happens to be the expected matching coverage for the {\sc Random} algorithm defined therein (which is roughly equivalent to our {\sc Greedy}). In the present work, we are interested in a different metrics: we establish convergence to a deterministic value of the average matching ratio, rather than a lower bound. 

So it is clear that our framework differ with that of \cite{Karp}. 
However, to gain some insights on how our algorithms based on the degree distribution stand against their counterparts on real graphs, we conducted the following study: we let $G$ be a randomly generated graph with 5000 nodes, from an upper triangular adjacency matrix that is specified as follows: all diagonal elements are 1, thereby insuring the existence of a perfect matching, and all upper elements are Bernoulli($p$) 
random variables, where $p$ is so that the graph has average degree, say, $5$. 
Such graphs provide the worst case scenarios for the framework in \cite{Karp}. Then, on the graph $G$ we run the exploration algorithms (Section \ref{sec:explo}) for both {\sc Greedy} and {\sc Minres}. In parallel, we extract  the degree distributions $(\mathbf d^+,\mathbf d^-)$ of the graph under consideration, on which we run both algorithms {\sc Greedy} and {\sc Minres} (joint construction of Section \ref{sec:CM}), thereby constructing another (multi-)graph having the same degree distribution. 
Our results can be summarized in Figure \ref{fig:expl}.

\begin{figure}[H]
\centering
	
			\includegraphics[width=0.8\linewidth]{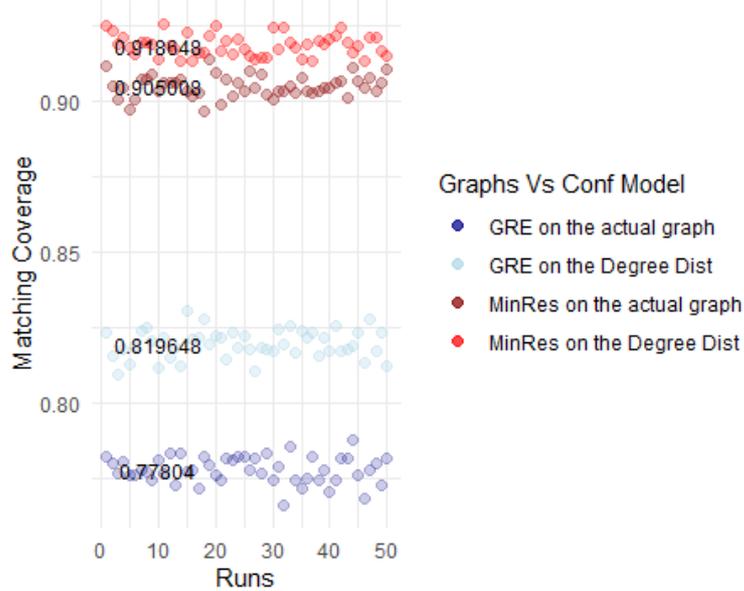}
			\caption{Exploration Vs Conf Model}
			\label{fig:expl}
\end{figure}

Figure \ref{fig:expl} shows that matching algorithms that are jointly constructed with the CM achieve a better matching coverage than 
on this particular graph. In other words, 
matching algorithms typically perform badly on this particular graph, with respect to a graph that is obtained as a uniform draw amongst 
graphs having the same degree distribution, unsurprisingly hinting at the influence of the graph topology on the considered matching algorithms. 
Simulations indicate that this influence is enhanced in {\sc Minres} with respect to {\sc Greedy}. Last, 
we observe again that {\sc Minres} produces a better matching coverage than {\sc Greedy} in all cases.

\subsection{Notes and Conclusion}

In this work, we introduced a procedure to approximate 
the matching coverage of local matching algorithms on graphs, using the hydrodynamic limit, in the large graph asymptotics, 
of a measure-valued Markov process representing the joint construction of the matching together with the graph itself, as a realization of the bipartite configuration model. 

Transposing the matching algorithms into the dynamics of the residual degree distributions of the considered graphs, allowed us to predict the results of the considered algorithms, with a remarkable accuracy. 
This results in a dramatic reduction of the problem complexity: as long as one is interested in the matching coverage of the algorithm under consideration, one only needs to keep track of the residual degree distribution, and {\em not} of the whole graph geometry.

As our simulations indicate, two natural and interesting problems arise: 
The first (and probably hardest) one is to quantify the influence of the graph topology 
 on the average matching coverage, for a given algorithm. 
%
The second concerns the optimality of {\sc Minres}. Indeed, among all local algorithms we addressed, 
{\sc Minres} always leads to better results. Is this criterion optimal for the matching coverage, and if so, in what probabilistic sense, and among which class of algorithms? 
Building on these two problems, we believe that the present work opens a promising line of research on the topic. 

\clearpage
\bibliographystyle{plain} 
\bibliography{bibli}

\end{document}